\documentclass[10pt,a4paper]{article}
\usepackage{a4wide}
\usepackage{float} 
\usepackage{authblk}

\usepackage[utf8]{inputenc}  
    
\usepackage{amsmath}

\usepackage{amssymb, amsfonts, amsbsy, latexsym}
\usepackage{graphicx}
\graphicspath{ {Images/} }
\usepackage{color}
\usepackage{amsmath}
\usepackage{hyperref}
\usepackage{xcolor}
\usepackage{amsthm}

\newtheorem{theorem}{Theorem}
\newtheorem{definition}[theorem]{Definition}
\newtheorem{proposition}[theorem]{Proposition}
\newtheorem{remark}[theorem]{Remark}
\newtheorem{lemma}[theorem]{Lemma}
\newtheorem{corollary}[theorem]{Corollary}
\newtheorem{Minimization problem}[theorem]{Minimization Problem}

\newcommand{\ip}[2]{\left\langle#1,#2\right\rangle}
\newcommand{\abs}[1]{\left|#1\right|}
\newcommand{\norm}[1]{\left\|#1\right\|}

\def\bT{\breve{T}}

\def\ts{\tilde{s}}
\def\tf{\tilde{f}}
\def\tg{\tilde{g}}

\def\tL{\tilde{L}}
\def\tR{\tilde{R}}
\def\tD{\tilde{D}}

\def\tpi{\tilde{\pi}}

\def\hh{\hat{h}}
\def\hg{\hat{g}}
\def\hf{\hat{f}}
\def\hs{\hat{s}}

\def\hL{\hat{L}}
\def\hT{\hat{T}}

\def\hpi{\hat{\pi}}
\def\hD{\hat{D}}

\def\Re{{\rm Re}}

\def\ln{{\rm ln}}

\def\w{\omega}

\def\e{\epsilon}

\def\l{\lambda}

\def\cF{\mathcal{F}}

\def\cH{\mathcal{H}}

\def\cL{\mathcal{L}}

\def\cU{\mathcal{U}}

\def\cW{\mathcal{W}}
\def\cS{\mathcal{S}}

\def\FF{\mathbb{F}}
\def\CC{\mathbb{C}}

\def\RR{\mathbb{R}}

\def\w{\omega}
\def\Tx{i\frac{\partial}{\partial\w}}

\def\ambig{V_{\frac{\hf}{\norm{\hf}_{\cW}}}({\frac{\hf}{\norm{\hf}_{\cS}}})}

\begin{document}

\title{Wavelet Design with Optimally Localized Ambiguity Function:  a Variational Approach}

\author[1]{Ron Levie}
\author[2]{Efrat Krimer Avraham}
\author[2]{Nir Sochen}
\affil[1]{Department of Mathematics, Ludwig-Maximilians-Universit{\"a}t M{\"u}nchen\newline {\small levie@math.lmu.de}}
\affil[2]{School of Mathematical Sciences, Tel Aviv University\newline {\small efratk2@mail.tau.ac.il, sochen@tauex.tau.ac.il}}

\date{ }

\maketitle

\begin{abstract}
In this paper, we design mother wavelets for the 1D continuous wavelet transform with some optimality properties. An optimal mother wavelet here is one that has an ambiguity function with minimal spread in the continuous coefficient space (also called phase space). Since the ambiguity function is the reproducing kernel of the coefficient space, optimal windows lead to phase space representations which are "optimally sharp." Namely, the wavelet coefficients have minimal correlations with each other. Such a construction also promotes sparsity in phase space. 
The spread of the ambiguity function is modeled as the sum of variances along the axes in phase space. In order to optimize the mother wavelet directly as a 1D signal, we pull-back the variances, defined on the 2D phase space, to the so called window-signal space. This is done using the recently developed wavelet-Plancharel theory. The approach allows formulating the optimization problem of the 2D ambiguity function as a minimization problem of the 1D mother wavelet. The resulting 1D formulation is more efficient and does not involve complicated constraints on the 2D ambiguity function. We optimize the mother wavelet using gradient descent, which yields a locally optimal mother wavelet.
\end{abstract}

\textbf{Keywords.} Continuous wavelet, mother wavelet design, uncertainty principle, Plancherel theorem, variational method


\section{Introduction}
\label{intro}
In this paper, we consider the 1D continuous wavelet transform (1D CWT in short). The 1D CWT is a special case of a generalized wavelet transform -- a signal transform which is based on taking the inner product of the input signal with a set of transformations of a window function. 
Some examples of generalized wavelet transforms are the short time Fourier transform
    (STFT) \cite{Grochenig2001_28}, the 1D CWT \cite{Grossman1984_29,Daub1992}, the Shearlet transform \cite{Guo2006}, the Curvelet
    transform \cite{CANDES2005_7}, and the dyadic wavelet transform \cite{Daub1992}. The first three examples are based on square integrable representations of a group. Such transforms are called \emph{continuous wavelet transforms} (see Subsection \ref{Continuous Wavelet Transforms}). A continuous wavelet transform is defined by the choice of the set of transformations, and the choice of the window function. In this paper, we introduce a systematic approach for choosing the window function.

    There are various approaches for window design in the literature.
	One line of work for window design is the classical method introduced by Daubechies in \cite{Daubechies_1988}, in which an orthogonal wavelet basis is designed, with a compactly supported window having some degree of vanishing moments. A high order of vanishing moments is linked with a sparser approximation of a (sufficiently regular) signal in phase space (\cite[Ch.7.1]{Mallat_2008}). Another family of approaches are adaptive methods, which also take into account the signal at hand and aim to maximize the correlation between the analyzing window and the signal (see, for example, \cite{Chapa_2001}).
	
	In this paper, we develop a method for choosing a window for the 1D continuous wavelet transform, based on an uncertainty minimization approach. The motivation comes from the use of wavelet transforms as means for measuring physical quantities of signals. %
	For example, the STFT measures
    the content of signals at different times and frequencies, and the 1D CWT measures times and scales. 
	For the measurements to be as accurate as possible, it is desirable to choose a window with minimal uncertainty with respect to the different physical quantities.

	In the following subsection, we review the evolution of the uncertainty minimization approach to window design in previous work. These past works will lead naturally to our proposed approach.

	\subsection{Window Design via Uncertainty Minimization}
	
	Consider a signal transform based on a square integrable representation $\pi(g)$ of a locally compact group $G$ in the Hilbert space $\cH$. Given an admissible window $f\in\cH$, the corresponding wavelet transform $V_f:\cH\rightarrow L^2(G)$ reads, for $s\in\cH$ and $g\in G$,
	\[V_f(s)(g) = \ip{s}{\pi(g)f}.\]

	\subsubsection{The Classical Approach to Wavelet Localization}
	
	For certain transforms, e.g., the 1D CWT and the Shearlet transform, there is a classical approach to window design. The idea is to generalize the classical localization framework of the STFT, and specifically to generalize the Heisenberg uncertainty principle. The general scheme can be described as follows.
	Consider a set of linearly independent infinitesimal self-adjoint generators $T_1,\dots, T_n$ of the
    group of transformations $\pi(G)$. Each $T_j$ generates a one-parameter group of transformations $e^{i\RR T_j}=\{e^{itT_j}\ |\ t\in\RR\}$, which is a subgroup of $\pi(G)$. Each $e^{i\RR T_j}$ is interpreted as a set of transformations that translates a certain physical quantity. For example, in the STFT, the one parameter transformation groups are translations, which change the physical quantity \emph{time}, and modulations, which change the quantity \emph{frequency}. Now, take the generators $T_1,\dots, T_n$ as the \emph{observables} of their respective physical quantities. Namely, use the variance of $f$ with respect to each $T_j$, $v_f(T_{j})$ (see Definition \ref{def: Observable localization}), 
     as a measure of the localization of $f$ with respect to the physical quantity underlying $e^{i\RR T_j}$. The overall uncertainty of the window function is then defined as the multiplication of the above variances $\prod_{j=1}^n v_f(T_j)$.
    This approach can be found in the literature, e.g.,  \cite{Ali2000_1,antoine2004_3}, and in papers, e.g.,  \cite{Dahlke2008_11,Dahlke1995_12}. In the classical case of the STFT, the generators are the frequency observable $T_{1}:f(x)\mapsto i\frac{\partial}{\partial x}f(x)$ and the time observable $T_{2}:f(x)\mapsto xf(x)$, and $v_f(T_1)$ and $v_f(T_2)$ measure the spread of $f$ in frequency and time respectively.

    For any pair of observables $T_k$ and $T_l$, the uncertainty principle poses a lower bound on the simultaneous concentration of $f$ with respect to the physical quantities underlying $T_k$ and $T_l$. Namely,
    \begin{equation}\label{uncertainty principle}
    v_f(T_k) v_f(T_l)\geq \frac{1}{4}\ip{f}{[T_k,T_l]f},
    \end{equation}
    where $[T_k,T_l]:= T_kT_l-T_lT_k$ denotes the \emph{commutator} of $T_k$ and $T_l$.
    The conventional approach to window design is to choose two observables of interest, and to find an equalizer of the above inequality -- an $f$ for which the inequality in (\ref{uncertainty principle}) becomes an equality. Such an $f$ is then declared as an uncertainty minimizers.  
    In \cite{Maass2010_39}, it was shown that this approach is erroneous, and that windows admitting equality are not  generally minimizers of the left-hand-side of (\ref{uncertainty principle}).
    Instead, in order to find a window with minimal uncertainty, one should minimize the uncertainty of $f$ using some variational method.
    
    \subsubsection{Towards a Coherent Approach to Wavelet Uncertainty}
    
    Even after applying variational methods to find the minimizer of the left-and-side of (\ref{uncertainty principle}), the above approach yields rather strange and counter-intuitive results (see e.g., \cite{Stark2011_49}).  
    The reason for that, as suggested in \cite{Levie2020}, is that the group generators $T_j$ are not appropriate for defining localization of the physical quantities underlying each transformation group $e^{i\RR T_j}$. The mistake comes from the fact that defining the observables as the generators works for the STFT ``by accident'', but this accident does not repeat in other transforms.  For the STFT case, the generator of time translations is the frequency obseravable $T_{1}:f(x)\mapsto i\frac{\partial}{\partial x}f(x)$, and the generator of modulations is the time observable $T_{2}:f(x)\mapsto xf(x)$. Hence, each of these two observables happen to be appropriate as a localization operator for the \emph{other} transformation group. 
     When combined in an uncertainty, the two generators happen to measure together the correct pair of physical quantities. However, this lucky accident does not generalize to other transforms, like the 1D CWT.  Hence, the generators cannot be taken as the observables in the general case.
	Instead, 
	\cite{Levie2014,Levie2020} suggested a distinction between the group generators and the operators that measure localization, i.e., the observables.

In \cite{Levie2020}, an interpretation of the uncertainty minimizer from the point of view of sparsity in phase space was proposed.
	Consider a signal $s$ for which a sparse representation in the 1D CWT time-scale phase space exists.
	Namely, there exists  a ``sparse phase space function''
	$F =\sum_{k=1}^{K} c_k\delta_{g_k}$
	such that $ s = V^*_f(F) = \sum_{k=1}^{K} c_k \pi(g_k) f$, where $\delta_{g_k}$ is the delta functional concentrated at $g_k\in G$.  
	Finding $F$ given only knowledge of $s$ is a difficult problem.  
	Conversely, calculating the continuous wavelet transform of $s$ is direct and simple, but does not reconstruct the sparse function $F$. To see this, consider the ambiguity function $V_f(f)\in L^2(G)$. It is known that synthesizing any function $H\in L^2(G)$, and immediately after analyzing it, $V_fV_f^* H$, is given by the convolution of $H$ with the ambiguity function $V_f(f)$ in phase space. Similarly, $V_f(s)$ is given by convolving $F$ with the ambiguity function $V_f(f)$. Therefore, finding an ambiguity function with minimal spread is desirable, as it results in a minimal blurring of $F$. For a well localized ambiguity function, the peaks $g_k$ of $F$  are well separated in $V_f(s)=V_fV_f^* F$, so they can be easily extracted from $V_f(s)$ without direct knowledge of $F$. 

	The above analysis was the motivation for the definition of the observables in \cite{Levie2020}. However, the localization measures in \cite{Levie2020} were not directly defined in phase space, not measuring the concentration of $V_f(f)$ directly, but rather based on surrogate localization measures of $f$ in the signal domain. The signal localization measures of $f$ were only shown to pose loose bounds on the decay of $V_f(f)$.
	Our goal, instead, is to define the uncertainty as the spread of $V_f(f)$ directly in phase space, and to pull-back this phase space-based definition to a formulation in the signal domain. The problem in pulling-back the uncertainty of $V_f(f)$ to a signal domain formulation in terms of $f$, is that the 1D CWT is not an isomorphism, but only an isometric embedding of the signal domain to $L^2(G)$.
	
	In a subsequent paper \cite{Levie2017}, the wavelet-Plancharel theory was developed. The wavelet-Plancherel theorem  establishes an isometric isomorphism between $L^2(G)$ and the so called \emph{window-signal space} -- the tensor product of the space of admissible windows with the space of signals.
	The wavelet-Plancherel theory introduces closed form formulas for pulling back phase space operations to window-signal space operations. In some cases, the theory results in formulations of operations applied on $V_f(s)$, as a combination of operations applied on $s$ and $f$ separately. This allows implementing 2D operations in phase space efficiently as 1D computations in the window and signal spaces.
	The wavelet-Plancherel theory is at the core of the current work, as it allows us to exactly formulate the localization measures of the ambiguity function $V_f(f)$ as 1D operations on $f$.
	
	\subsection{Related Work in Ambiguity Function Localization}
	The idea for localizing the ambiguity function was also studied in other papers. 
	In the context of the STFT, the ambiguity function is known to play an important role in the estimation of optimal localization properties in many different areas, e.g., for operator approximation by Gabor multipliers \cite{Doerfler2010} or for RADAR and coding applications \cite{Costas1984,Alltop1980,Herman2009}. 
	In \cite{feichtinger2012}, Feichtinger et al. presented a method for designing optimally localized windows in the time-frequency plane of the STFT. This is done by maximizing some measure of the concentration of the ambiguity function. Under general hypotheses and shape constraints, using a  variational method (recursive quadratizations), the algorithm converges to a window  whose ambiguity function has approximate (locally) optimal properties in the time-frequency domain.
	Our approach differs from this in several ways. First, our method is generic, and can be easily generalized to other continuous wavelet transforms based on square integrable representations. In addition,  the computations in our method are performed in the 1D window and signal spaces, in contrast to the 2D phase space, and are thus more efficient.

	\subsection{Our Contribution}
	
	We summarize our contribution as follows. 
	\begin{itemize}
	    \item We define the localization/uncertainty of the ambiguity function in the 2D phase space using variances along the axes.
	    \item
	    We utilize the wavelet-Plancherel theory to pull-back the 2D formulation  of uncertainty to a 1D formulation. The 1D formulation is more efficient and does not require complicated constraints for the 2D ambiguity function, since we optimize the window directly.
	    \item We show how to compute a locally optimal mother wavelet using calculus of variations and gradient descent.
	\end{itemize}
	The technique presented in this paper can be applied to other general wavelet transforms such as the Shearlet transform. The 1D wavelet transform can be regarded as a prototype to the localization machinery, that can be generalized to every transform for which the wavelet-Plancharel theory is applicable -- the so called \emph{semi-direct product wavelet transforms} (see \cite[Definition 8]{Levie2017}). 
	
	\subsection{Outline}\label{outline}
	In Section \ref{sec 1D CWT}, we recall the 1D continous wavelet transform and its wavelet-Plancharel theory. 
	 In Section \ref{Localization in Wavelet Analysis}, we discuss the observable-based localization theory of the 1D CWT, and define our notion of wavelet uncertainty as the localization of the ambiguity function in phase space. 
	 In Section \ref{sec uncertainty pullback} we develop formulas for the pull-back of the the phase space uncertainty to the window-signal space. We use these pull-back formulas to obtain closed-form-formulas of the phase space uncertainty 
	as a combination of window and signal localization measures. 
	In order to find a local uncertainty minimizer via a variational method, in Section \ref{calculus of variations} we develop a general theory for calculus of vriations of observables.
	In Section  \ref{sec minimization}, we use this theory  to derive the Euler-Lagrange equations for the uncertainty minimizer. Finally, in Section \ref{sec implementation}, we implement a gradient descent method  to numerically estimate an optimal window of the 1D wavelet transform.
	
	\section{The 1D Wavelet-Plancherel Theory}\label{sec 1D CWT}

	The wavelet-Plancharel framework was developed in  \cite{Levie2020,Levie2017}.
	This theory encompasses a broad class of wavelet transforms which are based on square integrable representations.  
	In this section, we recall briefly some of the fundamentals of the theory for the special case of the 1D  continuous wavelet transform.  We begin by recalling general continuous wavelet transforms.

	\subsection{Continuous Wavelet Transforms}
	\label{Continuous Wavelet Transforms}
	
	Let $G$
	be a locally compact group that we call \emph{phase space}, and let $d\mu$ denote the left Haar measure in $G$. Consider a Hilbert space $\cH$ that we call the \emph{signal space}, and let $\pi:G\to \cU(\cH)$ be a square integrable representation of $G$. Here, $\cU(\cH)$ denotes the space of unitary operators on $\cH$.
	The square integrable assumption means that there is a signal $f\in \cH$, that we call a \emph{window function} or a \emph{mother wavelet}, such that the mapping
	\[V_f:s\mapsto  \ip{s}{\pi(\cdot)f}_{\cH}\]
	maps $\cH$ to $L^2(G;d\mu)$.
Under this construction, $V_f$ is called a \emph{continuous wavelet transform}. Continuous wavelet transforms have many useful analytic properties, e.g., they are isometric embeddings to the coefficient space $L^2(G)$, they have closed form reconstruction formulas, and they have orthogonality relations that allow the formulation of the wavelet-Plancherel theory. Instead of introducing these properties in their generality, in the next subsections we recall the theory for the special case of the 1D CWT.

	\subsection{The 1D Continuous Wavelet Transform}
	\label{1D CWT}
	The 1D CWT is based on a representation of the 1D affine group ``$\alpha x+\beta$'' 
	\cite{Daub1992}. The affine group is defined as $G = \RR' \ltimes \RR$, where $\RR'=\RR\setminus\{0\}$, and where $\ltimes$ denotes the semi-direct product of groups. The group law is 
	\[(\alpha,\beta)(\alpha',\beta') = (\alpha\alpha', \beta+\alpha\beta').\]
	The wavelet system is generated by the  representation $\pi$ acting on $L^2(\RR)$ via
	\[(\pi(\alpha,\beta)f)(x) = |\alpha|^{-\frac{1}{2}}f\big(\alpha(x-\beta)\big).\]
	
	To use the wavelet-Plancherel theory of semi-direct product wavelet transforms, the subgroups of $G$ corresponding to the coordinates $\alpha$ and $\beta$ should be represented as translation groups. We thus choose a different parameterization of $G$, substituting $\alpha$ with $e^a$. Note that this parametrization only contains  positive dilations for $a\in\RR$. To include also negative dilations, we consider the reflection group $\{ -1,1\}$, and represent $G$ as $(\RR \times \{-1,1\})\ltimes\RR$. Indeed, $\RR'\sim \RR \times \{-1,1\}$, where $\alpha>0$ is equivalent to $(e^{ln(\alpha)},1)=(e^a,1)$ and  $\alpha<0$ to $(e^{ln(-\alpha)},-1)= (e^a,-1)$. The group rule of this new parameterization of $G$ is given by
	\[(a,b,c)\cdot(a',b',c') = (a+a',b+ce^{a}b',cc').\]
	With this parameterization of $G$, the representation, which by abuse of notation is still denoted by $\pi$, takes the form
    \begin{equation}\label{Pi}  
    (\pi(a,b, c)f)(x) = e^{-\frac{a}{2}}f(ce^{-a}(x-b)).
    \end{equation}
    
    The window $f\in L^2(\RR)$ is assumed to meet the \emph{admissibility condition}, which takes the following form in the frequency domain 
    \begin{equation}\label{admissibility}
    \int_{\RR}\frac{1}{|\w|}\Big|\hf(\w)\Big|^2 d\w<\infty.  
    \end{equation}
    The \emph{wavelet transform} associated to the window $f$ maps signals $s$ in $L^2(\RR)$ to functions $V_f(s)$ defined over $G$ by
	\[V_{f}(s)(a,b,c)=\ip{s}{\pi(a,b,c) f},\]
	where $\ip{\cdot}{\cdot\cdot}$ denotes the standard inner product in $L^2(\RR)$. The wavelet transform is also called the \emph{analysis operator}. 
	
	The admissibility condition ensures that the mapping $V_f$ is an isometric embedding of $L^2(\RR)$ to  $L^2(G, d\mu_G(a,b,c))$ up to a constant, where the weighted Lebesgue measure $d\mu_G(a,b,c)=e^{-a} dadb$ is the left Haar measure of $G$. More accurately, integrating measurable functions $F:G\rightarrow\CC$, under the $(a,b,c)$ pamaterization,  is given by
	\[\int_G F(a,b,c)d\mu_{G}(a,b,c) = \sum_{c\in\{-1,1\}}\iint_{\RR^2} F(a,b,c)e^{-a} dadb.\]
	We write in short $L^2(G)$ instead of $L^2(G, d\mu_G(a,b,c))$.
	
	The wavelet synthesis operator is defined as the adjoint $V_f^*$ of the wavelet transform. We have, for every $F\in L^2(G)$,
	\begin{equation}
	\label{Wsynth}
	 V_f^* F = \int_G F(a,b,c)\ \pi(a,b,c)f \ d\mu_G(a,b,c),  
	\end{equation}
	where the integration in (\ref{Wsynth}) is a weak integral.
	The synthesis operator is the pseudo inverse of the analysis operator, up to constant, and we have the reconstruction formula
	\[s = \Big(\int_{\RR} \abs{\hf(\w)}^2\frac{1}{\abs{\w}}d\w\Big)^{-1} V_f^* V_f(s) \]
	
	The isometric embedding property can be derived from the more general \emph{orthogonality relation}, given as follows.
	For every pair of signals $s_1, s_2\in L^2(\RR)$ and admissible vectors $f_1,f_2\in L^2(\RR)$, 
	\begin{equation}\label{eq orthogonality}
 \ip{V_{f_1}(s_1)}{
V_{f_2}(s_2)} =\int_{\RR} \overline{\hf_1(\w)}\hf_2(\w)\frac{1}{\abs{\w}}d\w\ip{s_1}{s_2}_{L^2(\RR)}.
	\end{equation}

	\subsection{The Wavelet-Plancharel Theory}
	\label{W-P thm}
	The wavelet transform $V_f : L^2(\RR) \mapsto L^2(G)$ is an isometric embedding of $L^2(\RR)$ into $L^2(G)$, and is not surjective. 
	Consequently, it is impossible to pull-back most phase space operators isometrically. 
	The wavelet-Plancharel theorem aims to remedy this problem. This is done by embedding the signal space $L^2(\RR)$ in a larger space, called the window-signal space, and canonically extending the wavelet transform to a isometric isomorphism between the larger signal space and $L^2(G)$. In this subsection we recall the wavelet-Plancherel theory for the special case of the 1D CWT \cite{Levie2017}.
	
    \subsubsection{The Window-Signal Space}
    \label{W-S space}
    
    We begin by defining the window-signal space in the case of the 1D CWT. To avoid carrying the Fourier transform in our formulas, we formulate the theory directly in the frequency domain.  
    The window-signal space is defined by combining two spaces, called the \emph{window space} and the \emph{signal space}.
    For the 1D CWT, the signal space is $\cS:=L^2(\RR)$, and the window space $\cW$ is defined to be the Hilbert space of measurable functions $\hat{f}:\RR\rightarrow\RR$ satisfying the admissibility condition (\ref{admissibility}), with the inner product, 
    \[\ip{\hf_1}{\hf_2}_{\cW} = \int_{\RR}\hf_1(\w)\overline{\hf_2(\w)}\frac{1}{\abs{\w}} d\w.\]
    We use the notations $\cW$ and $\cS$, either as a superscript or a subscript, to denote that the computations are done with respect to the window or signal inner product.  For example, $\norm{\hf}_{\cW}$ denotes the norm of $\hf$ in $\cW$, and $\norm{\hf}_{\cS}$ the norm in $\cS$.
    Note that not every function in $\cW$ is in $L^2(\RR)$, and that the space of admissible vectors is $\cW\cap \cS$.
    
    Next we define the window-signal space.
    \begin{definition}\label{window-signal}
    $ $
    \begin{itemize}
        \item 
        The \emph{window-signal space} $\cW\otimes \cS$ is defined as the tensor product of $\cW$ with $\cS$. Namely, 
    \[\cW\otimes \cS: = L^2(\RR^2;\frac{1}{\abs{\w'}}d\w'd\w),\]
    where $\w'$ denotes the window variable and $\w$ denotes the signal variable.  
    \item
    The \emph{tensor product operator} $(\hf,\hs)\mapsto \hf\otimes\hs$ maps window-signal pairs from $W\times S$ to $W\otimes 
 S$, 
 where the \emph{simple function} $\hf\otimes\hs$ is defined as
 \begin{equation}\label{simple tensor}
 \hf\otimes\hs(\w',\w) = \overline{\hf(\w')}\hs(\w)
 \end{equation}
    \end{itemize}
    
    \end{definition}
 
    Note that not every function in $\cW\otimes\cS$ is simple. The window-signal space $\cW\otimes\cS$ is the linear closure of the space of simple functions. Namely, every $F\in \cW\otimes\cS$ can be approximated by a finite sum of simple functions.
    Also, note that for any two simple functions, the inner product in $\cW\otimes\cS$ satisfies  
    \begin{equation}\label{simple ip}
    \ip{\hf_1\otimes\hs_1}{\hf_2\otimes\hs_2} =\overline{\ip{\hf_1}{\hf_2}}_{\cW}\ip{\hs_1}{\hs_2}_{\cS}.    
    \end{equation}
    
	\subsubsection{The Wavelet-Plancharel Transform}\label{W-P transform}
	
	The \emph{wavelet-Plancharel transform} extends the wavelet transform to an isometric isomorphism $V : \cW\otimes\cS\to L^2(G)$ as follows. For simple functions, the wavelet-Plancharel transform is defined by 
	\[V(\hf\otimes\hs) = V_{\hf} (\hs),\]
	and this definition extends by linear closure to any $F\in\cW\otimes\cS$.
	From (\ref{simple ip}) and the orthogonality relation (\ref{eq orthogonality}), it follows that $V$ is an isometric embedding into $L^2(G)$. The next theorem from \cite[Example 24]{Levie2017} states that $V$ is in fact an isometric isomorphism for the 1D continuous wavelet transform. 
    \begin{theorem}[The Wavelet-Plancharel Theorem]\label{wavelet-Plancharel}
    The wavelet-Plancherel transform 
    $V$ is an isometric isomorphism between $\cW\otimes \cS$ and $L^2(G)$. 
    \end{theorem}
    
    In \cite[Equation (13)]{Levie2017} is was shown that the wavelet-Plancherel transform can be written by the following explicit formula. For any $F\in \cW\otimes\cS$,
    \[V(F)(a,b,c) = \int_{\RR}  e^{2\pi i b \w} e^{\frac{1}{2}a} F(ce^{a}\w,\w)d\w.\]
    The inverse wavelet-Plancherel transform, $V^*$, also has a closed form formula, presented in \cite[Section 2.2.4]{Levie2017}. We skip this explicit inversion formula as we do not use it in this paper.

    \section{Localization in Wavelet Analysis}
    \label{Localization in Wavelet Analysis}
    
    In this section, we recall the localization theory of wavelet transforms, presented in \cite{Levie2014,Levie2020,Levie2017}, for the specific case of the 1D CWT. We use the approach to define our notion of wavelet uncertainty.
    
    A well known fact from wavelet theory (see, e.g., \cite{fuhr2005}) is that the image space $V_f(L^2(\RR))$ of the wavelet transform is a reproducing kernel Hilbert space. The kernels of the image space 
	are the translations of the \emph{ambiguity function}   $V_f(f) \in L^2(G)$. 
    Namely, for a phase space function $F\in V_f(L^2(\RR))$,
    it holds, for every $(a,b,c)\in G$,  
    \[F(a,b,c) = \ip{F}{\lambda(a,b,c) V_f(f)}.\]
    Here, $\lambda(a,b,c):L^2(G)\to L^2(G)$ is the left translation in $L^2(G)$, defined 
    for functions $H\in L^2(G)$ by
      $\lambda(a,b,c)H(a',b',c')= H\big((a,b,c)^{-1}(a',b',c')\big)$. 
    We can view $\overline{V_f(f)}$ as the ``point spread function'' of the space $V_f(L^2(\RR))$, whose spread describes the ``blurrines'' 
    of the space.
    This motivates the main goal in our paper -- to make the ambiguity function as localized as possible. 
    In order to define precisely how to measure the locality of the ambiguity function, 
    we need to introduce special operators called observables.
    
    \subsection{Observables}

    An \emph{observable} in a separable Hilbert space $\cH$ is a self-adjoint or unitary operator. Observables are seen as entities that define and measure physical quanities. One way in which observable measure their underlying quantities is through expected values and variances.
    
    \begin{definition}\label{def: Observable localization}
    Let $\bT$ be an observable (i.e., a self-adjoint or unitary operator) in the Hilbert space $\cH$. 
    The \emph{expected value} and the \emph{variance} of a normalized vector $f\in {\rm Dom}(\bT)$, with respect to $\bT$, are defined to be, respectively,
    \begin{equation}\label{exp}
        e_f(\bT)= \ip{\bT f}{f},
    \end{equation}
    \begin{equation}\label{var}
        v_f(\bT)= \norm{(\bT- e_f(\bT)) f}^2.
    \end{equation}
    The expected value and variance are also called the (first order and second order) \emph{moments} of $f$ with respect to $\bT$.
    \end{definition}
    
    The simplest examples of observables are multiplicative operators.
	The \emph{phase space scale observable} is the operator $A$, defined for functions $F:G\rightarrow\CC$ by
	\begin{equation}
	    \label{Aobs}
	    A F(a,b,c) = a F(a,b,c).
	\end{equation}
    The operator $A$ is self-adjoint in the domain
    \begin{equation*}
    {\rm Dom}(A) = \{F\in L^2(G)\mid {\rm the~function~} (a,b,c)\mapsto aF(a,b,c) {\rm ~is~in~} L^2(G)\}.
    \end{equation*}
    We interpret $A$ as an entity that measures scale as follows. For a normalized $F\in L^2(G)$, we can think about $\abs{F(a,b,c)}^2$ as the weight of the point $(a,b,c)$, for each $(a,b,c)\in G$. Computing $e_F(A)=\ip{AF}{F}$ weighs the scale value $a$ of every point $(a,b,c)$ by $\abs{F(a,b,c)}^2$, and computes the weighted average of the scale values $a$. Thus, $e_F(A)$ is seen as the center of mass of $F$ along the axis $a$ in phase space. Similarly, the variance $v_F(A)$ is seen as the spread of $F$ along $a$ about its center of mass.
    Equivalently, the \emph{phase space time observable} $B$, is defined by
    \begin{equation}
	    \label{Bobs}
	    B F(a, b,c) = bF(a, b,c),
	\end{equation}
    and is self-adjoint over the domain
    \begin{equation*}{
    {\rm Dom}(B) = \{F\in L^2(G)\mid {\rm the~function~} (a,b,c))\mapsto bF(a,b, c){\rm ~is~in~} L^2(G)\}.}
    \end{equation*}

Now, we combine the moments of the ambiguity function $V_f(f)$ to define one uncertainty measure of mother wavelets. 
    Two natural ways to combine $v_{V_f(f)}(A)$ and $v_{V_f(f)}(B)$ are by  
    the multiplication $v_{V_f(f)}(A)v_{V_f(f)}(B)$ or by the sum $v_{V_f(f)}(A)+v_{V_f(f)}(B)$. Since minimizing the sum of the variances ensures both small area and small radius of the domain that contains most of the energy of the ambiguity function, we base the uncertainty on sum. 
    \begin{definition}\label{def:phase uncertainty}
    The \emph{phase space uncertainty} associated to the window $\hf$ is defined to be
    \begin{equation}\label{phase uncertainty}
        \cL(\hf) =v_{K_{\hf}}(A)+v_{K_{\hf}}(B).
    \end{equation}
    where $K_{\hf}:=\ambig$ is the normalized ambiguity function. The domain of $\cL$ is defined to be
    \[{\rm Dom} (\cL) =\{\hf\in\cW\cap\cS|~\hf\neq 0{\rm ~and~} K_{\hf}\in {\rm Dom} (A)\cap {\rm Dom} (B)\}.\]
    \end{definition}

         For an observable $\bT$, when the vector $f\in \cH$ is not normalized, we still use the notations $e_f(\bT)$ and $v_f(\bT)$ defined in (\ref{exp}) and (\ref{var}). In this case, $e_f(T)$ and $ v_f(T)$ are no longer interpreted as center of mass and spread.
     In the context of the window and signal spaces,
      we denote by $e^{\cW}_{\hf}(\bT)$ and $v^{\cW}_{\hf}(\bT)$  the expected value and variance of $\hf$ with respect to an observable $\bT$ in the space  $\cW$, and similarly denote by $e^{\cS}_{\hf}(\bT)$ and $v^{\cS}_{\hf}(\bT)$ the localization measures in the signal domain. For example
     $e^{\cW}_{\hf}(\bT)= \ip{T \hf}{\hf}_{\cW}$.

     Our next goal is to formulate the pull-back of the moments of $A$ and $B$ to the window-signal space, so $\cL(\hf)$ can be computed efficiently and directly using signal and window operations on the window function $\hf$.
    For the pull-back formulas of $\cL(\hf)$, we first define observables corresponding to localization in wavelet analysis directly in the signal domain. For that, we need to define rigorously what is meant by the statements ``the dilation group changes scale,'' and  ``translations change time.''
    
    \subsection{Transforms Associated with the CWT}
    \label{1D CWT transforms}
    In this subsection, we recall the transformation subgroups of $\pi(G)$ that are associate with the physical quantities underlying the 1D CWT: \emph{time} and \emph{scale}. We also recall two useful signal transforms in the context of wavelet analysis. 
    
    Note that $\pi(a,b,c)$ is the composition of three operators in $L^2(\RR)$, 
	\[\pi(a,b,c) =L(b) D(a) R(c),\]
	where $L$, $D$ and $R$ denote the translation, dilation, and reflection operators, respectively,
	\[L(b)f(x) = f(x-b), \quad D(a)f(x) = e^{-\frac{a}{2}}f(e^{-a} x), \quad R(c)f(x) = f(cx).\]
	We call the parameter $b\in\RR$ \emph{time}, call $a\in\RR$ \emph{scale}, and $c\in\{-1,1\}$ \emph{direction}.
	
	Since the window and signal spaces, $\cW$ and $\cS$, are defined directly as function spaces over the frequency line, we next 
	formulate $\pi$ in the frequency domain. For an operator $T$ in the time domain, we denote the pull-back to the frequency domain by $\hT := \cF T \cF^{-1}$. Here, $\cF$ denotes the Fourier transform. 
	Translation by $b$ in the time domain takes the form of modulation in frequency,
	\[\hL(b)\hf(\omega) := e^{- i b\omega}\hf(\omega).\]
	Dilation by scale $a$ in frequency is
	\[\hD (a)\hf(\w) = e^{\frac{1}{2} a }\hf (e^{a}\w).\]
	Reflection in the frequency domain stays reflection, 
	\[{\hat{R}(c)\hf(\w)=\hf(c\w)}.\]
	Overall, 
	\begin{equation}\label{hPi}
	    (\hpi(a,b,c)\hf)(\omega) = e^{- i b\omega}e^{\frac{1}{2} a }\hf (ce^{a}\w).
	\end{equation}

	One way to formulate rigorously the statement "dilations change scale," is to transform the signal space $L^2(\RR)$ into another space $L^2(Y)$, where dilations operate as translations. In this space, we can call points $\sigma\in Y$ scales.
    Next, we recall the transform from the frequency domain $L^2(\RR)$ to the so called \emph{scale space}. 
    Define the subspaces of \emph{positively supported (i.e. analytic) and negatively supported signals},
    \[L_{A^{+}}^2(\RR) = \{f\in L^2(\RR)|~ support(\hf)\subseteq\RR_+\},\]
    \[L^2_{A^{-}}(\RR)= \{f\in L^2(\RR)|~ support(\hf)\subseteq\RR_-\}.\] 
    When we define these spaces directly in the frequency domain, $L_{A^{+}}^2(\RR)$ and $L^2_{A^{-}}(\RR)$ take the forms $L^2(\RR_{+})$ and $L^2(\RR_{-})$ respectively, where
    \[L^2(\RR_{+}):= \{\hf\in L^2(\RR)|~ support(\hf)\subseteq\RR_+\},\]
    \[L^2(\RR_{-}):= \{\hf\in L^2(\RR)|~ support(\hf)\subseteq\RR_-\}.\]
    We identify the frequency domain $L^2(\RR)$ as the direct sum $L^2(\RR_{+})\oplus L^2(\RR_{-})$, and consider the two \emph{warping transforms} 
    \begin{equation}
    \begin{split}
    \label{warping}
    & W_{+}:L^2(\RR_{+})\to L^2(\RR),\quad W_{-}:L^2(\RR_{-})\to L^2(\RR)\\
    & W_{+} \hf_{+} (\sigma)=\tf_{+}(\sigma)=e^{-\sigma/2}\hf_+( e^{-\sigma}), \quad W_{-} \hf_{-} (\sigma)=\tf_{-}(\sigma)=e^{-\sigma/2}\hf_{-}(- e^{-\sigma}),\\
    \end{split}
    \end{equation}
    defined separately on the positively supported and negatively supported components, as in \cite{Levie2020}.  
	The inverse warping transforms are given by
	\[ W_{+ }^{-1}\tf_{+}(\w) =
	\left\{
	\begin{array}{cc}
	   \w^{-\frac{1}{2}}\tf_{+}(-\ln(\w))  &  {\ \  \rm if\ } \w>0 \\
	   0  & {\ \  \rm if\ } \w\leq 0
	\end{array}
	\right. = \hf_{+}(\w),\]
	\[ W_{- }^{-1}\tf_{-}(\w) =
	\left\{
	\begin{array}{cc}
	   (-\w)^{-\frac{1}{2}}\tf_{-}(-\ln(-\w))   &  {\ \  \rm if\ } \w<0 \\
	   0  & {\ \  \rm if\ } \w\geq 0
	\end{array}
	\right.
	= \hf_{-}(\w).\]
	
	The scale transform is defined as the application of the two warping transforms on the two frequency components of signals $\hf$.  
	
	\begin{definition}
	\label{scale_Def}
	$ $
	\begin{itemize}
	\item
	We define the \emph{scale space} as $L^2(\RR)^2$ , and parameterize its variable as $(\sigma,{\rm sign})$. Namely, for  $\tilde{f}=(\tf_{+}, \tf_{-})\in L^2(\RR)^2$, we define $\tilde{f}(\sigma,{\rm sign}) = \tf_{{\rm sign}}(\sigma)$. We call the variable $(\sigma,{\rm sign})$  \emph{scale}.
	The inner product in $(L^2(\RR))^2$ is defined as 
	 \[\ip{(\tf_{+}, \tf_{-})}{(\tg_{+}, \tg_{-})}_{(L^2(\RR))^2} := \ip{\tf_{+}}{\tg_{+}}_{L^2(\RR)} + \ip{\tf_{-}}{\tg_{-}}_{L^2(\RR)},\]
	 for every $(\tg_{+}, \tg_{-}), (\tf_{+}, \tf_{-})$ in $(L^2(\RR))^2$. 
	    \item 
	    The scale transform $U:L^2(\RR)\rightarrow L^2(\RR)^2$ is defined by its frequency pull-back as
	$\cF U\cF^{-1}= W_{+}\oplus W_{-}$, where $W_{+}$ and  $W_{-}$ are the warping transforms (\ref{warping}).
	Namely, for a time signal 
	$f\in L^2(\RR)$, with $\hf=\hf_{+}+\hf_{-}$ as the decomposition of $\hf$ to its positive and negative supports, we have 
	\[\cF U\cF^{-1}(\hf_{+}+\hf_{-})(\sigma,{\rm sign}) =\tf_{{\rm sign}}(\sigma).\]
	\end{itemize}
	\end{definition}
	We denote the image of $f\in L^2(\RR)$ under the scale transform in short by $Uf=\tf$. 
	It is easy to see that the scale transform is a unitary operator. Namely, 
	\begin{equation}
	    \norm{f}_{L^2(\RR)}=\norm{\tf}_{(L^2(\RR))^2}
	\end{equation}
	
	Translation and dilation in scale space take the forms
	\[U L(b) U^{-1} \tf(\sigma,{\rm sign})=\tL(b)\tf(\sigma,{\rm sign}) = e^{-i b e^{-\sigma}}\tf(\sigma,{\rm sign})  ,\]
	\begin{equation}
	\label{eq_scale_trans}
	    U D(a) U^{-1} \tf(\sigma,{\rm sign})=\tD(a)\tf(\sigma,{\rm sign}) = \tf(\sigma-a,{\rm sign}).
	\end{equation}
	Reflection in scale space is 
	\[U R(c) U^{-1} \tf(\sigma,{\rm sign})=\tR(c)\tf(\sigma,{\rm sign}) = \tf(\sigma, c\cdot{\rm sign}).\]
	Overall, 
	\begin{equation}\label{eq tpi}
	U \pi(a,b,c) U^{-1} \tf(\sigma,{\rm sign})=
	  \tpi(a,b,c)\tf(\sigma,{\rm sign}) = e^{-i b e^{-\sigma}}\tf(\sigma-a,c\cdot{\rm sign})
	\end{equation}
	
	From (\ref{eq_scale_trans}), we see that dilations operate as translations in the scale space under the scale transform, which justifies the terms \emph{scale space} and \emph{scale transform}. Namely, given a normalized function $\tf\in L^2(\RR)^2$, where $\abs{\tf(\sigma,{\rm sign})}^2$ is interpreted as the weight or probability of the scale $(\sigma,{\rm sign})$, $U D(a) U^{-1}$ translates the scale distribution by $a$.

Since both the Fourier transform and the scale transform are isometries, we have
	\[V_f(s) = V_{\hf}(\hs) =  V_{\tf}(\ts), \]
	where  
	$V_{\hf}(\hs)$ 
	and
	$V_{\tf}(\ts)$
	denote by abuse of notation
	\[V_{\hf}(\hs)(a,b,c) = \ip{\hs}{\hpi(a,b,c)\hf}_{L^2(\RR)},\quad V_{\tf}(\ts)(a,b,c) = \ip{\ts}{\tpi(a,b,c)\tf}_{\big(L^2(\RR)\big)^2}. \]
	
	Lastly, the \emph{time transform} is defined as the identity in the time domain $L^2(\RR)$, as $L(b)$ is already represented as translations in $L^2(\RR)$.

	\subsection{Observables in the Signal Domain}
	
	Next, we recall two observables of scale and time, defined in the signal space. These signal space observables will be used in the explicit formula of the pull-back of the uncertainty $\cL(\hf)$.  The signal space time and scale observables are defined as multiplicative operators in the time and the scale domains, as defined in the previous subsection. Since we focus in this paper on frequency space formulations, we define all operators directly in the frequency domain, assuming that $\cS=L^2(\RR)$ is the signal space in the frequency domain. Here, the time transform is given by $\cF^{-1}$. We define in this setting the scale transform by $\hat{U}=\cF U \cF^{-1}=W_+\oplus W_-$ (see Definition \ref{scale_Def}). In the following, we denote by $\breve{Y}$ the self-adjoint operator in $L^2(\RR)$ defined by
	\[\breve{Y}:f(y)\mapsto yf(y),\]  
    on the domain 
    \begin{equation*}
    {\rm Dom}(\breve{Y})=\{f\in L^2(\RR)\mid y\mapsto yf(y)\in L^2(\RR)\}.
    \end{equation*}
    
    \begin{definition}
    Let $\cS=L^2(\RR)$ be the frequency domain.
    \begin{enumerate}
    \item The \emph{time signal space observable}, $\bT_x$, is the operator that multiplies the signal in the time domain by its $x$ coordinate, i.e.,
    \[\bT_x = \cF\breve{Y}\cF^{-1}.\]  
    \item The \emph{scale signal space observable}, $\bT_{\sigma}$, is the operator that multiplies the signal in the scale space by its scale coordinate $\sigma$
    \[\bT_{\sigma} = \hat{U}^{-1}(\breve{Y}\oplus\breve{Y})\hat{U}.\]
    Here,
    \[(\breve{Y}\oplus\breve{Y})\tf(\sigma,{\rm sign}) = \sigma\tf(\sigma,{\rm sign}).\]
    \end{enumerate}
    \end{definition}
    
    The next claim is direct.
    \begin{proposition}
    \label{dom_obse_clas}
    \begin{enumerate}  
        \item The time observable is given by $\bT_x=i\frac{\partial}{\partial\w}$ 
        on the domain ${\rm Dom}(\bT_x)$ of all $\hf\in L^2(\RR)$ such that $\hf$  is absolutely continuous in $[\alpha,\beta]$ for every $\beta>\alpha$, and $\w\mapsto i\hf'(\w)$ is in $L^2(\RR)$. Moreover, $\bT_x$ is self-adjoint.
        \item The scale observable is given by the multiplicative operator  
    \[\bT_{\sigma}\hf(\w) = -\ln(\abs{\w})\hf(\w), \]
    on the domain 
    \[{\rm Dom}(\bT_\sigma)= \{\hf\in L^2(\RR)\mid \w\mapsto -{\rm ln}(\abs{\w})\hf(\w)\in L^2(\RR)\}.\]
    The scale observable is self-adjoint.
    \end{enumerate}
    \end{proposition}

    The two signal space observables $\bT_x$ and $\bT_{\sigma}$
    were used to define the uncertainty measure in \cite{Levie2014} as a combination of $v_{\hf}(\bT_x)$ and $v_{\hf}(\bT_{\sigma})$. In our theory, the uncertainty is defined via the phase space observables, instead of the 
     signal space observables. %
   The signal space observables appear in our formulations only as a result of using the wavelet-Plancherel theory to pull-back $\cL(\hf)$.

    \subsection{The Uncertainty Minimization Problem}\label{sec minimization2}

The signal space observables $\bT_x$ and $\bT_{\sigma}$ 
    are also used  
    to restrict the search space of the window, when minimizing the uncertainty $\cL(\hf)$ of Definition \ref{def:phase uncertainty}. Precisely, the window $\hf$ is assumed to satisfy $e_{\hf}(\bT_{\sigma})=e_{\hf}(\bT_x)=0$. For motivation, consider the role of $\hf$ as an analyzing function, i.e., as a mean for probing the signal content at different times and scales. From this point of view, $e_{\hf}(\bT_{\sigma})$ is interpreted as the scale location of the window ${\hf}$, and $e_{\hf}(\bT_x)$ as its time location. A window ${\hf}$ with $e_{\hf}(\bT_{\sigma})=e_{\hf}(\bT_x)=0$ will be localized in time about $0$, and have a rate of one oscillation per time unit. Now, consider the following transformation formulas from \cite[Proposition 21]{Levie2014} 
        \begin{equation}\label{eq: e_D(a)f(Ta)}
            e_{\hat{D}(a){\hf}}(\bT_{\sigma}) = e_{\hf}(\bT_{\sigma})+a,
            \end{equation}
        \begin{equation} \label{eq: e_L(b)f(Tb)}
                e_{\hat{L}(b) \hf}(\bT_x) = e_{\hf}(\bT_x)+b. 
         \end{equation}
    As a result of (\ref{eq: e_D(a)f(Ta)}) and (\ref{eq: e_L(b)f(Tb)}),  by picking a window with zero scale and time expected values, we assure that the transformed window $\hat{D}(a)\hf$ measures the scale $a$, and $\hat{L}(b)\hf$ measures the time $b$.
    Similarly, $\hat{\pi}(a,b,c)\hf$ is
    a scale-time atom that probes the scale-time pair $(a,b)$ (see \cite[Formula (58)]{Levie2017}).

To conclude this section, we formulate the minimization problem, following the above discussion.

\begin{Minimization problem}
\label{Min_prob1}
Let $\cL$ be the uncertainty from definition \ref{def:phase uncertainty}. Find
\[{\rm Arg} \min_{\hf\in {\rm Dom}(\cL)}\cL(\hf) ,\]
subject to 
$e_{\hf}(\bT_x) = e_{\hf}(\bT_{\sigma})= 0.$
\end{Minimization problem}

	\section{The Pull-Back of the Phase Space Uncertainty}
	\label{sec uncertainty pullback}
    Let $K_{\hf} =\ambig $ be the normalized ambiguity function. 
    In this section, we formulate the variances of the 2D ambiguity function $v_{K_{\hf}}(A)$ and $v_{K_{\hf}}(B)$ in terms of the 1D window function $\hf$, and hence formulate the phase space uncertainty via signal and window space computations.
	 Note that $\hf$ plays two roles in the ambiguity function, both as the window and as the signal. In addition, observe that windows and signals are treated differently in the wavelet-Plancherel theory. Therefore, it is beneficial to first study more general variances of the form $v_{V_{\hf}(\hs)}(A)$ and $v_{V_{\hf}(\hs)}(B)$ for general normalized signals $\hs\in\cS= L^2(\RR)$ and $\hf\in \cW =L^2(\RR,\frac{1}{\abs{\w}}d\w)$, which need not be identical.
	 
	\subsection{The Pull-Back of Phase Space Observables}\label{observabeles pullback}
	In this section, we recall formulas for the pull-back of the time and scale observables in the window-signal space.  
	In ~\cite[Propositions 27 and 33]{Levie2017}, there is a general formula for calculating the pull-back of a broad class of multiplicative operators for diffeomorphism the so called geometric wavelet transforms.  
	Using this general formula we can compute our case of the 1D CWT. For completeness, we provide direct computations for the pull-back of the time and scale phase space observables in \ref{appendixA}. 
	
	\begin{definition} 
	The \emph{pull-backs} of the phase space observales $A$ and $B$ of (\ref{Aobs}) and (\ref{Bobs}), to the window-signal space, are defined as follows.
	\begin{enumerate}
	    \item The pull-back $\bT_a:= V^*AV$ is defined over the domain  
	    \[{\rm Dom} (\bT_a) = \{F \in \cW\otimes\cS\mid VF\in {\rm Dom}(A)\}.\]
 	    \item The pull-back $\bT_b:= V^*BV$ is defined defined over 
	    \[{\rm Dom} (\bT_b) = \{F \in \cW\otimes\cS\mid VF\in {\rm Dom}(B)\}.\]
	\end{enumerate}
	\end{definition}
	The following proposition summarizes the results from \cite[Proposition 33 and Section 7.8]{Levie2017} for the 1D CWT. 
	
	\begin{proposition}
	\label{prop: pull-back on simple}
	\begin{enumerate}
	    \item 
	    The operator $\bT_a$ is self-adjoint, and defined by
	\[\bT_a F (\w,\w') = \big(\ln(\abs{\w'})-\ln(\abs{\w})\big)F (\w,\w')\]
	over the domain
	\[{\rm Dom}(\bT_a)= \{F\in\cW\otimes\cS\ |\ \big(\ln(\abs{\w'})-\ln(\abs{\w})\big)F (\w,\w')\in L^2\big(\RR^2; \frac{1}{\abs{\w'}}d\w' d\w\big) \}.\]
	\item
	 The operator $\bT_b$ is self-adjoint, and defined by
	\[\bT_b F (\w,\w') = \big(i\frac{\partial}{\partial\w}+i\frac{\w'}{\w}\frac{\partial}{\partial\w'}\big)F (\w,\w')\]
	over the domain ${\rm Dom}(\bT_b)$ of all functions $F\in \cW\otimes\cS$ such that their restriction to every compact interval in almost every (with respect to $\kappa\in\RR$) line of the form
	\[l_{\kappa} = \{ (\w,\kappa\w) \ |\ \w\in\RR \}\]
	is absolutely continuous, and 
	\[(\w',\w)\mapsto \big(i\frac{\partial}{\partial\w}+i\frac{\w'}{\w}\frac{\partial}{\partial\w'}\big)F (\w,\w') \in L^2\big(\RR^2; \frac{1}{\abs{\w'}}d\w' d\w\big).\]
	\end{enumerate}
	\end{proposition}
	In the explicit formula of $\bT_b$, note that the lines $l_{\kappa}$ are the integral lines of the translation group $\{e^{i t\bT_b}\}_{t\in\RR}$, and hence for absolutely continuous functions $F$ in every interval along these integral lines, $\big(i\frac{\partial}{\partial\w}+i\frac{\w'}{\w}\frac{\partial}{\partial\w'}\big)F (\w,\w') $ is well defined. 
	
	Next, we consider the pull-back formulas of Proposition \ref{prop: pull-back on simple} for simple functions $F=\hf\otimes \hs$. To be able to derive closed-form formulas that disentangle the computations on $\hs$ from those on $\hf$, we restrict the domain of $\bT_b$.
	
	For the rest of this paper, we follow the following convention when writing operators. By abuse of notation, we write multiplicative operators $Q$ that operate on function $f$ by
	\[Qf(x)=q(x)f(x),\]
	where $q$ is some function, by $q(x)$. We write operators $W$ that operate of function $f$ by
	\[Wf(x)=q(x) \frac{\partial}{\partial x}f(x),\]
	where $q$ is some function, by $q(x)\frac{\partial}{\partial x}$.
	
	\begin{corollary}
	\label{cor_simp_obs}
	Let $\hs\in\cS$ and $\hf\in\cW$.
	\begin{enumerate}
	    \item 
	    Suppose that $\hs$ is in the domain
	    \[{\rm Dom}_{\cS}(\hT_a)= \{\hs\in L^2(\RR)\mid \w\mapsto {\rm ln}(\abs{\w})\hs(\w)\in L^2(\RR)\},\]
	    and
	    $\hf$ is in the domain
	    \[{\rm Dom}_{\cW}(\hT_a)= \{\hf\in L^2(\RR)\mid \w'\mapsto {\rm ln}(\abs{\w'})\hf(\w')\in L^2(\RR;\frac{1}{\abs{\w'}}d\w')\}.\]
	    Then $\hf\otimes\hs\in {\rm Dom}(\bT_a)$, and
	    \begin{equation}
	    \label{simple_scale_obs}
	        \bT_a (\hf\otimes \hs)= -\hf\otimes\Big(\ln(\abs{\w})\hs\Big) + \Big(\ln(\abs{\w'})\hf\Big)\otimes \hs.
	    \end{equation}
	    \item
	    Let ${\rm Dom}_{\cS}(\hT_b)\subset \cS$ be defined as
	    \[{\rm Dom}_{\cS}(\hT_b)= {\rm Dom}(\bT_x) \cap  \{\hs\in\cS\ |\ \w\mapsto \frac{1}{\w}\hs(\w)\in L^2(\RR)\} ,\]
	    where ${\rm Dom}(\bT_x)$ is defined in Proposition \ref{dom_obse_clas} (as the natural domain of $i\frac{\partial}{\partial \w}$).
	    Let ${\rm Dom}_{\cW}(\hT_b)\subset \cW$ denote the domain of all windows $\hf\in\cW$ such that $\hf$  is absolutely continuous in $[\alpha,\beta]$ for every $\beta>\alpha$, and $\w'\mapsto i\w'\hf'(\w')$ is in $L^2(\RR;\frac{1}{\abs{\w'}}d\w')$.
	    If $\hs\in {\rm Dom}_{\cS}(\hT_b)$ and $\hf\in{\rm Dom}_{\cW}(\hT_b)$, then
	    $\hf\otimes\hs\in {\rm Dom}(\bT_b)$, and
	    \begin{equation}
	    \label{simple_time_obs}
        \bT_b(\hf\otimes\hs) = \hf\otimes i\frac{\partial}{\partial\w}\hs - i\w' \frac{\partial}{\partial\w'}\hf\otimes\frac{1}{\w}\hs.
        \end{equation}
	\end{enumerate}
	\end{corollary}

    For future calculations, it is beneficial
    to formulate $\bT_a$ and $\bT_b$ 
    for simple vectors 
    as sums of orthogonal simple vectors.  
    By the inner product formula of simple functions in $\cW\otimes\cS$ (see (\ref{simple ip})), it is enough to make the windows orthogonal. Hence, we reformulate (\ref{simple_scale_obs}) and (\ref{simple_time_obs}) as
	\begin{equation}\label{T_a}
	 \bT_a (\hf\otimes \hs)= \hf\otimes \Big(-\ln(\abs{\w}) +e^{\cW}_{\hf}\big(\ln(\abs{\w'})\big) \Big)\hs + \Big(\ln(\abs{\w'}) - e^{\cW}_{\hf}\big(\ln(\abs{\w'})\big)\Big)\hf\otimes \hs,  
	\end{equation}
    \begin{equation}\label{bT_b}
    \bT_b (\hf\otimes \hs) = \hf\otimes \big( i\frac{\partial}{\partial\w}-e^{\cW}_{\hf}(i\w'\frac{\partial}{\partial\w'})\frac{1}{\w} \big)\hs  +\big(  -i\w' \frac{\partial}{\partial\w'}+e^{\cW}_{\hf}(i\w'\frac{\partial}{\partial\w'}) \big)\hf \otimes \big( \frac{1}{\w}\hs\big).    
    \end{equation}
    Here, recall that $e^{\cW}_{\hf}(T)=\ip{T\hf}{\hf}_{\cW}$ denotes the expected value of $\hf$ with respect to an observable $T$ in the space $\cW$.

	\subsection{The Pull-Back of Scale Localization Measures}
	\label{scale localization}
	In this subsection we formulate $e_{\hf\otimes \hs}(\bT_{a})$ and $v_{\hf\otimes \hs}(\bT_{a})$ as a combination of signal and window expected values and variances.
	\begin{proposition}
	\label{prop:pull_scale}
	Let $\hs\in {\rm Dom}_{\cS}(\hT_a)$ and $\hf\in{\rm Dom}_{\cW}(\hT_a)$ (see Proposition \ref{cor_simp_obs})
	such that  $\norm{\hf}_{\cW}=\norm{\hs}_{\cS}=1$. Then the following holds.
	\begin{enumerate}
	    \item The expected value of $\hf\otimes \hs$ with respect to $\bT_a$ is 
	\[e_{\hf\otimes \hs}(\bT_{a})
    	 = -e^{\cS}_{\hs}\big(\ln(\abs{\w})\big) + e^{\cW}_{\hf}\big(\ln(\abs{\w'})\big) .\]
	    \item The variance of $\hf\otimes \hs$ with respect to $\bT_a$  is
	    \[v_{\hf\otimes\hs}(\bT_{a}) 
	   \quad=v^{\cS}_{\hs}\big(\ln(\abs{\w})\big) + v_{\hf}^{\cW}\big(\ln(\abs{\w'})\big).\]
	\end{enumerate}
	\end{proposition}
	
	\begin{proof}
	\begin{enumerate}
	    \item
	    Calculate
	\begin{equation}\label{mean A}
	\begin{split}
    	&e_{\hf\otimes \hs}(\bT_{a})\\
    	&\quad = \ip{\bT_{a}(\hf\otimes \hs)}{\hf\otimes \hs}_{\cW\otimes\cS}\\
    	&\quad= \ip{\hf\otimes \Big(-\ln(\abs{\w}) +e^{\cW}_{\hf}\big(\ln(\abs{\w'})\big) \Big)\hs }{\hf\otimes\hs}_{\cW\otimes\cS}\\
    	& \quad\quad +\ip{ \Big(\ln(\abs{\w'}) - e^{\cW}_{\hf}\big(\ln(\abs{\w'})\big)\Big)\hf\otimes \hs}{\hf\otimes\hs}_{\cW\otimes\cS}\\
    	&\quad = \ip{\hf}{\hf}_{\cW}\ip{\Big(-\ln(\abs{\w}) +e^{\cW}_{\hf}\big(\ln(\abs{\w'})\big) \Big)\hs}{\hs}_{\cS}
    	\\
    	& \quad\quad+\ip{\Big(\ln(\abs{\w'}) - e^{\cW}_{\hf}\big(\ln(\abs{\w'})\big)\Big)\hf}{\hf}_{\cW}\ip{\hs}{\hs}_{\cS}\\
    	&\quad = -e^{\cS}_{\hs}\big(\ln(\abs{\w})\big) + e^{\cW}_{\hf}\big(\ln(\abs{\w'})\big) .
	\end{split}
	\end{equation}
	
	Note that the third and fourth equalities in (\ref{mean A}) follows from linearity, the definition (\ref{simple ip}) of inner product of simple vectors in $\cW\otimes\cS$, and the fact that 
	$\ip{\ln(\w)\hf + e^{\cW}_{\hf}(-\ln(\w))\hf}{\hf}_{\cW}=0.$
	    \item We start by computing 
	\[
	\begin{split}
	\big(\bT_{a} - e_{\hf\otimes \hs}(\bT_{a})\big) (\hf\otimes\hs)
	&  =
	\hf\otimes \Big(-\ln(\abs{\w}) +e^{\cS}_{\hs}\big(\ln(\abs{\w})\big) \Big)\hs\\ 
	&\quad + \Big(\ln(\abs{\w'})  -e^{\cW}_{\hf}\big(\ln(\abs{\w'})\big)\Big)\hf\otimes \hs.
	\end{split}
	\]
	Hence, the variance is 
	\begin{equation}\label{var A}
	\begin{split}
	v_{\hf\otimes\hs}(\bT_{a}) 
	   &\quad= \norm{(\bT_{a} - e_{\hf\otimes \hs}(\bT_{a}))(\hf\otimes\hs)}^2_{\cW\otimes\cS}\\
	   &\quad = \norm{\hf}_{\cW}^2\norm{\Big(-\ln(\abs{\w}) +e^{\cS}_{\hs}\big(\ln(\abs{\w})\big) \Big)\hs}_{\cS}^2\\  
	   &\quad+\norm{\Big(\ln(\abs{\w'})  -e^{\cW}_{\hf}\big(\ln(\abs{\w'})\big)\Big)\hf}_{\cW}^2\norm{\hs}_{\cS}^2\\
	   &\quad = v^{\cS}_{\hs}\big(\ln(\abs{\w})\big) + v_{\hf}^{\cW}\big(\ln(\abs{\w'})\big),
	   \end{split} 
	\end{equation}	
	where the second equality in (\ref{var A}) follows from linearity, the inner product formula of simple functions in $\cW\otimes\cS$ (\ref{simple ip}), and by the orthogonality of the two windows $\hf$ and $\Big(\ln(\w') - e^{\cW}_{\hf}(\ln(\w'))\Big)\hf$.  %
	\end{enumerate}
	\end{proof}
	
	\subsection{The Pull-Back of Time Localization measures}
	\label{time localization}
	In this subsection we formulate $e_{\hf\otimes \hs}(\bT_{b})$ and $v_{\hf\otimes \hs}(\bT_{b})$ as a combination of signal and window expected values and variances.
	 
	\begin{proposition}
	\label{prop:pull_time}
	Let $\hs\in {\rm Dom}_{\cS}(\hT_b)$ and $\hf\in{\rm Dom}_{\cW}(\hT_b)$ (see Proposition \ref{cor_simp_obs})
	such that  $\norm{\hf}_{\cW}=\norm{\hs}_{\cS}=1$. Then the following holds.
	\begin{enumerate}
	    \item  The expected value of $\hf\otimes \hs$ with respect to $\bT_b$ is
	    \begin{equation}\label{mean B} 
	    e_{\hf\otimes \hs}(\bT_{b})
    	 =e^{\cS}_{\hs}(i\frac{\partial}{\partial\w})  - e^{\cW}_{\hf}(i\w'\frac{\partial}{\partial\w'})e^{\cS}_{\hs}(\frac{1}{\w}).
	\end{equation}
	    \item 
	    The variance of $\hf\otimes \hs$ with respect to $\bT_b$ is 
	    \begin{equation}
	     v_{\hf\otimes\hs}(\bT_{b})
	   =v^{\cS}_{\hs}(i\frac{\partial}{\partial\w}-e^{\cW}_{\hf}(i\w'\frac{\partial}{\partial\w'})\frac{1}{\w}) +v_{\hf}^{\cW}(i\w'\frac{\partial}{\partial\w'})\norm{\frac{1}{\w}\hs}_{\cS}^2.  
	    \end{equation}
	\end{enumerate}
	\end{proposition}
	\begin{proof}
	\begin{enumerate}
	    \item Compute 
	    \begin{equation}
    \begin{split}
	    &e_{\hf\otimes \hs}(\bT_b)\\ 
	    &\quad=\ip{\hf\otimes \big( i\frac{\partial}{\partial\w}\hs-e^{\cW}_{\hf}(i\w'\frac{\partial}{\partial\w'})\frac{1}{\w}\hs \big)}{\hf\otimes \hs}\\  
	    &\quad\quad + \ip{
	    \Big(  -i\w' \frac{\partial}{\partial\w'}\hf
	    +e^{\cW}_{\hf}(i\w'\frac{\partial}{\partial\w'})\hf \Big) \otimes \big( \frac{1}{\w}\hs\big)}{\hf\otimes \hs}\\ 
	   &\quad= e^{\cS}_{\hs}(i\frac{\partial}{\partial\w}) \ -\ e^{\cW}_{\hf}(i\w'\frac{\partial}{\partial\w'})e^{\cS}_{\hs}(\frac{1}{\w})
	\end{split}
\end{equation}
	    \item
	    Compute  
	\[
	\begin{split}
	 & (\bT_b - e_{\hf\otimes \hs}(\bT_b))(\hf\otimes\hs)
	\\ & = 
	\hf\otimes \Big( i\frac{\partial}{\partial\w}\hs - e^{\cW}_{\hf}(i\w'\frac{\partial}{\partial\w'})\frac{1}{\w}\hs \Big)\\
	&\quad + \Big(  -i\w' \frac{\partial}{\partial\w'}\hf + e^{\cW}_{\hf}(i\w'\frac{\partial}{\partial\w'})\hf \Big) \otimes  \frac{1}{\w}\hs(\w)\\
	&\quad -\Big( e^{\cS}_{\hs}(i\frac{\partial}{\partial\w})  -  e^{\cW}_{\hf}(i\w'\frac{\partial}{\partial\w'}) e^{\cS}_{\hs}(\frac{1}{\w})\Big) (\hf\otimes\hs)\\
	& =
	\hf\otimes \Big(i\frac{\partial}{\partial\w}\hs - e^{\cW}_{\hf}(i\w'\frac{\partial}{\partial\w'})\frac{1}{\w}\hs  - \big( e^{\cS}_{\hs}(i\frac{\partial}{\partial\w})  -  e^{\cW}_{\hf}(i\w'\frac{\partial}{\partial\w'}) e^{\cS}_{\hs}(\frac{1}{\w})\big)\hs\Big)\\
	&\quad + \Big(  -i\w' \frac{\partial}{\partial\w'}\hf + e^{\cW}_{\hf}(i\w'\frac{\partial}{\partial\w'})\hf \Big) \otimes \Big( \frac{1}{\w}\hs\Big).
	\end{split}
	\]
	Therefore, 
	\begin{equation}\label{var B}
	 \begin{split}
	    v_{\hf\otimes\hs}(\bT_b) 
	    &\quad = \norm{(\bT_b - e_{\hf\otimes \hs}(\bT_b))(\hf\otimes\hs)}_{\cW\otimes\cS}^2\\
	    &\quad =v^{\cS}_{\hs}(i\frac{\partial}{\partial\w}-e^{\cW}_{\hf}(i\w'\frac{\partial}{\partial\w'})\frac{1}{\w}) +v_{\hf}^{\cW}(i\w'\frac{\partial}{\partial\w'})\norm{\frac{1}{\w}\hs}_{\cS}^2.
	\end{split}   
	\end{equation} 
	 
	\end{enumerate}
	\end{proof}
	
	\subsection{A Pull-Back Formula for Phase Space Uncertainty}

We are now ready to write the phase space uncertainty $\mathcal{L}(\hf)$ of Definition \ref{def:phase uncertainty} in terms of 1D localization measures of $\hf$ via the pull-back formulas. To be able to write explicit formulas based on Propositions \ref{prop:pull_scale} and \ref{prop:pull_time}, we restrict the domain of $\cL$ to
\begin{equation}
\label{Dom_unc_f}
  {\rm Dom}_{\cW\cap\cS}(\cL):={\rm Dom}_{\cS}(\bT_a)\cap{\rm Dom}_{\cW}(\bT_a)\cap{\rm Dom}_{\cS}(\bT_b)\cap{\rm Dom}_{\cW}(\bT_b)\cap\{\hf\in\cW\ |\ \hf\neq 0\}.  
\end{equation}

\begin{proposition}
The domain ${\rm Dom}_{\cW\cap\cS}(\cL)$ is the set of all $0\neq\hf\in L^2(\RR)$ such that 
\begin{enumerate}
\item
$\hf$ is absolutely continuous in every compact interval, and
\[i\frac{\partial}{\partial\w}\hf\in L^2(\RR) {\rm \quad and} \quad  i\sqrt{\abs{\w}}\frac{\partial}{\partial\w}\hf\in L^2(\RR).\]
\item
$\w\mapsto \frac{1}{\w}\hf(\w)\in L^2(\RR)$
\item
$\w\mapsto \ln(\abs{\w})\hf(\w)\in L^2(\RR)$.
\end{enumerate}
\end{proposition}

\begin{proof}
All of the constraints in ${\rm Dom}_{\cW\cap\cS}(\cL)$ are dominated by the requirements (1)--(3). On the other hand, (1)--(3) follow the constraints in ${\rm Dom}_{\cW\cap\cS}(\cL)$.
\end{proof}

\begin{proposition}\label{cor cL}
Let $\hf\in {\rm Dom}_{\cW\cap\cS}(\cL)$ satisfy  $e^{\cS}_{\hf}\big(\ln(\abs{\w})\big) = e^{\cS}_{\hf}\big(\Tx\big)= 0$. Then,  $\hf\in{\rm Dom} (\cL)$ and
the phase space uncertainty $\cL(\hf)$ of $\hf$ is 
\begin{equation}\label{cL}
\begin {split}
\cL(\hf) 
&= v^{\cS}_{\frac{\hf}{\norm{\hf}_{\cS}}}\big(\ln(\abs{\w})\big) + v_{\frac{\hf}{\norm{\hf}_{\cW}}}^{\cW}\big(\ln(\abs{\w'})\big)\\
&\quad+v^{\cS}_{\frac{\hf}{\norm{\hf}_{\cS}}}\big(i\frac{\partial}{\partial\w} 
\big) 
+ v^{\cW}_{\frac{\hf}{\norm{\hf}_{\cW}}}\big(i\w'\frac{\partial}{\partial\w'}\big)\frac{\norm{\frac{1}{w}\hf}_{\cS}^2}{\norm{\hf}^2_{\cS}}.
\end {split}
\end{equation}

\end{proposition} 

We hence restrict the minimization problem of $\cL$ to the domain ${\rm Dom}_{\cW\cap\cS}(\cL)$, and use the explicit formula (\ref{cL}) in the uncertainty optimization problem.

\section{Calculus of Variations of Observables}
\label{calculus of variations}

Our goal is to develop a calculus of variation approach for solving the Minimization Problem \ref{Min_prob1}, with $\cL(\hf)$ given in the form of Proposition \ref{cor cL}. In this section, we start by developing general calculus of variations for localization measures based on observables.

\subsection{General Calculus of Variations}

To formulate the general calculus of variations framework, we consider (non-linear) functionals $S$ which are defined in separable Hilbert spaces $\cH$. A functional is simply a function from a domain in $\cH$ to the scalar field of $\cH$. The domain of $S$, denoted by ${\rm Dom} (S)$, is assumed to be a dense linear subspace of $\cH$, possibly excluding zero. 
Note that our uncertainty functional $\cL$ of (\ref{cL}), defined on ${\rm Dom}_{\cW\cap\cS}(\cL)$, is such a functional. 
We start by presenting the definition of Gateaux differential and the  variation in Hilbert spaces. 
  
\begin{definition}\label{def: Gataux }
Let $\cH$ be a separable Hilbert space over the filed $\FF$, which is $\RR$ or $\CC$, and let $S:{\rm Dom}(S)\rightarrow\FF$ be a (generally non-linear) functional. Suppose that $\rm Dom(S)\cup\{0\}\subset\cH$ is a dense linear subspace of $\cH$.
\begin{itemize}
    \item 
    The \emph{Gateaux differential} of $S$ at $f\in {\rm Dom}(S)$ in the direction $h \in {\rm Dom}(S)$ is defined to be
\[dS(f;h) := \lim_{\FF\ni t\to 0} \frac{S(f +th) - S(f)}{t},\]
if the limit exists. 
\item\label{var def}
If there exists a vector $\frac{\delta S}{\delta f}=\frac{\delta S}{\partial f}(f)\in \cH$ 
such that
\[dS(f;h)  = \ip{\frac{\delta S}{\delta f}}{h}_{\cH}\]
for every $h\in{\rm Dom}(S)$, 
$\frac{\delta S}{\delta f}(f)$ is called the \emph{variation} of $S$ at $f$.
\end{itemize}
\end{definition}

Next, we define the o-notation for functions, functionals, and operators.
\begin{definition}
\label{def_o-not}
 Let $\cH$ be a Hilbert space over the field $\FF$, where $\FF$ is $\RR$ or $\CC$. 
\begin{enumerate}
    \item For functions $e_1, e_2:\FF\to\FF$, we say $e_1(h)=o(e_2(h))$ as $h\rightarrow 0$ if
    \[\lim_{h\to 0}\frac{|e_1(h)|}{|e_2(h)|} = 0.\]
    \item Let $S_1:{\rm Dom}(S_1)\rightarrow\CC$ and $S_2:{\rm Dom}(S_2)\rightarrow\CC$ be (non-linear) functionals in $\cH$, such that ${\rm Dom}(S_j)\cup\{0\}$ is a dense linear subspace of $\cH$, for $j=1,2$.  Suppose that ${\rm Dom}(S_1)\subseteq{\rm Dom}(S_2)$.
    We say $S_1(h)=o(S_2(h))$ as $h\to 0$ if 
    \[\lim_{t\to 0}\frac{\abs{S_1(th)}}{\abs{S_2(th)}} = 0\]
    for every $h\in {\rm Dom}(S_1)$.
    \item Let $E: {\rm Dom}(E)\rightarrow\cH$ be (non-linear) operator in $\cH$, where ${\rm Dom}(E)\cup\{0\}$ is a dense subspace of $\cH$.  Let $S:{\rm Dom}(S)\rightarrow\CC$ be a functional with ${\rm Dom}(S)\cup\{0\}$ a dense subspace of $\cH$, such that ${\rm Dom}(E)\subset {\rm Dom}(S)$. We say $E(h)=o(S(h))$ as $h\to 0$ if 
    \[\lim_{t\to 0}\frac{\norm{E(th)}}{\abs{S(th)}} = 0\]
    for every $h\in {\rm Dom}(E)$.
\end{enumerate}
\end{definition} 
The following proposition is a useful tool when computing and proving the existence of variations.

\begin{proposition}\label{prop var existance}
Let $\cH$ be a Hilbert space over the field $\FF$, where $\FF$ is $\RR$ or $\CC$. Let $S:{\rm Dom}(S)\rightarrow\FF$ be a functional over the domain ${\rm Dom}(S)\subset\cH$. Suppose that ${\rm Dom}(S)\cup\{0\}$ is a dense linear subspace of $\cH$. Let $f\in{\rm Dom}(S)$ and $g\in\cH$. The following conditions are equivalent.
\begin{itemize}
    \item There exists a variation of $S$ at $f$ and $\frac{\delta}{\delta f}S(f)=g$.
    \item  It holds
    \[S(f +h) - S(f)   = \ip{g}{h}_{\cH}+ o(\norm{h}),\]
    as ${\rm Dom}(S)\ni h \rightarrow 0$. 
\end{itemize}
\end{proposition}
Calculus of variations helps us find local minima of functionals by locating their stationary points.
The next theorem is a version of Fermat's extreme value theorem for calculus of variations.  

\begin{theorem}[Fermat's Theorem for Stationary Points]
Let $\cH$ be a Hilbert space over the field $\FF$, where $\FF$ is $\RR$ or $\CC$. Let $S:{\rm Dom}(S)\rightarrow\RR$ be a functional. Suppose that ${\rm Dom}(S)\cup\{0\}$ is a dense linear subspace of $\cH$. If $f_0$ is a local minimum of $S$, and $S$ has a variation at $f_0$, then $\frac{\delta S}{\delta f}(f_0)=0$. 
\end{theorem}

\subsection{Realification}
Our goal is to minimize $\cL$ using calculus of variations. Note that a global minimum is a set theoretic notion that does not depend on additional structure endowed upon the set, like a vector space structure or inner product. Hence, we have the freedom to choose any vector space and inner product structure on the domain of the uncertainty $\cL$. 
To allow the use of calculus of variations tools for minimization, we require a vector space structure for which the variations of the uncertainty $\cL$ exist.
A natural choice is to take either $\cS$ or $\cW$ as the Hilbert space structure of the domain of $\cL$. However, it can be shown that the variations of $\cL$ do not exist with respect to neither of these two Hilbert spaces.
Fortunately, as we show constructively in the next sections, the variations do exist with respect to the \emph{realifications} of the spaces $\cW$ and $\cS$. 
\begin{definition}[\cite{constantin_2016}]
\label{DefReal}
The \emph{realification} of a complex Hilbert space $\cH$ (with the inner product $\ip{\cdot}{\cdot}_{\cH}$) is the real Hilbert space $\cH^{\RR}$, which is defined as follows.
\begin{itemize}
    \item The space $\cH^{\RR}$ is defined to be equal to $\cH$ as a set.
    \item The space $\cH^{\RR}$ is a vector space over the field $\RR$, with the same vector addition as in $\cH$, and the multiplication by scalars in $\cH^{\RR}$ is defined as the multiplication in $\cH$, restricted to real scalars.
    \item The inner product in $\cH^{\RR}$ is defined to be
    \begin{equation}\label{eq realification}
    \ip{\cdot}{\cdot\cdot}_{\cH^{\RR}} := {\rm Re}(\ip{\cdot}{\cdot\cdot}_{\cH}),
    \end{equation}
    where ${\rm Re}(\cdot)$ is the real part of a complex number.
\end{itemize}
\end{definition}
Equation (\ref{eq realification}) indeed defines an inner product in $\cH^{\RR}$, and $\cH^{\RR}$ is a Hilbert space. 

\begin{remark}
   The realificated signal space $\cS^{\RR}$ is the space of measurable complex valued functions $\hs:\RR\rightarrow\CC$ with the inner product
   \[\ip{\hs_1}{\hs_2}_{\cS^{\RR}} = {\rm Re}\Big(\int_{\RR} \hs_1(x)\overline{\hs_2(x)} dx\Big).\]
   Note that $\cS^{\RR}$ is not a space of real valued functions.
\end{remark}
   
\subsection{Variations with Normalized Vectors}
\label{normalized windows variations}
The moments in the uncertainty $\cL$ are computed with respect to normalized windows $\frac{\hf}{\norm{\hf}_{\cS}},\frac{\hf}{\norm{\hf}_{\cW}}$. Namely, for each moment $S$ that appears in $\cL$, there exists a functional $H$  (densely defined in $\cS$), such that
$S(\hf) = H(\frac{\hf}{\norm{\hf}})$, where the normalization is with respect to either $\cS$ or $\cW$.
In this subsection, we derive useful formulas that aid in computing variations of functionals that involve normalized windows.
 
We start by giving a definition for the variation of (possibly) non-linear operators that map vectors to vectors in general Hilbert spaces.

\begin{definition}\label{def op var}
Let $\cH$ be a separable Hilbert space. 
Let ${\rm Dom}(g)\subseteq\cH$ be a set such that ${\rm Dom}(g)\cup \{0\}$ is a dense linear subspace of $\cH$.
Let $g:{\rm Dom}(g)\to\cH$ be a (generally non-linear) operator.  
If there exists a bounded linear operator $T$ in $\cH$ 
 such that
\[ g(f+h)- g(f) = T (h) +o(\norm{h}), \]
where the $o$ notation is with respect to ${\rm Dom}(g)\ni h \rightarrow 0$,
then $T$ is called the \emph{variation} of $g$ at $f$. In this case, we denote $\frac{\delta g}{\delta f}(f)=T$.
\end{definition}
The following lemma shows the existence of a variation for the normalizing operator $f\mapsto\frac{ f}{\norm{f}}$, and gives an explicit formula for $\frac{\delta\frac{ f}{\norm{f}}}{\delta f}$.
\begin{lemma}\label{lemma normalization var}
Let $\cH$ be a real separable Hilbert space, with the inner product $\ip{\cdot}{\cdot\cdot}$ and norm $\norm{\cdot}$. The variation of the normalizing operator $f\mapsto\frac{f}{\norm{f}}$, with respect to $\cH$, exists at every non-zero $f$, and satisfies
\begin{equation}\label{eq normalization var}
  \frac{\delta \frac{ f}{\norm{f}}}{\delta f} = \frac{1}{\norm{f}}I - f\otimes\frac{f}{\norm{f}^3},  
\end{equation}
where $I$ is the identity operator, and $f\otimes\frac{f}{\norm{f}^3}$ is the rank-one operator  $v\mapsto \ip{v}{f}\frac{f}{\norm{f}^3}$.
\end{lemma}
\begin{proof}
In this proof, $\approx$ denotes equality up to $o(\norm{h})$. If the following, we use the fact that the scalar field of $\cH$ is $\RR$. 
Since $h$ is taken asymptotically small, we may assume $\norm{h}<\norm{f}$, and 
\[
\begin{split}
\frac{f + h}{\norm{ f + h}}&\quad\approx \frac{f+h}{\sqrt{\ip{f}{f} + 2 \ip{f}{h}}}\\ 
&\quad= \frac{f+h}{\norm{f}\sqrt{1 + 2 \ip{\frac{f}{\norm{f}^2}}{h}}}\\
&\quad\approx(f+h)\big(\frac{1}{\norm{f}}-\ip{\frac{f}{\norm{f}^3}}{h}\big) 
\quad\approx\frac{f}{\norm{f}}  + \frac{h}{\norm{f}} - \ip{\frac{f}{\norm{f}^3}}{h}f. 
\end{split}
\] 
By Definition \ref{def op var}, the lemma follows.   

\end{proof}

To be able to compute the variation of a functional composed on the normalization operator, in the next lemma we first formulate a version of the chain rule. 
\begin{lemma}\label{lemma chain rule functional-operator}
Let $\cH$ be a Hilbert space over the field $\FF$, where $\FF$ is $\RR$ or $\CC$. Let $S:{\rm Dom}(S)\rightarrow\FF$ be a functional over the domain ${\rm Dom}(S)\subset\cH$. Suppose that ${\rm Dom}(S)\cup\{0\}$ is a dense linear subspace of $\cH$. Suppose that there exist a functional $H:{\rm Dom}(H)\rightarrow\FF$, where ${\rm Dom}(H)\cup\{0\}$ is a dense linear subspace of $\cH$,  and an operator $T:{\rm Dom}(S)\rightarrow {\rm Dom}(H)$, such that for every $f\in{\rm Dom(S)}$
\[S(f) = H(T(f)).\]
 In addition, suppose that the variations $\frac{\delta H}{\delta t}(T(f))$ and $\frac{\delta T}{\delta f}(f)$ exist for some $f\in{\rm Dom(S)}$. Then, the variation $\frac{\delta S}{\delta f}(f)$ exists and satisfies
\[\frac{\delta S}{\delta f}(f) = \Big(\frac{\delta T}{\delta f}(f)\Big)^*\Big(\frac{\delta H}{\delta t}\big(T(f)\big)\Big),\]
where the $\Big(\frac{\delta T}{\delta f}(f)
\Big)^*$ is the adjoint operator of $\frac{\delta T}{\delta f}(f)$.
\end{lemma}
\begin{proof} Compute 
\[
\begin{split}
S(f+h)-S(h) &\quad= H(T(f+h))-H(T(f))\\
&\quad=\ip{\frac{\delta H}{\delta t}\big(T(f)\big)}{T(f+h)-T(f)} +\e(h)\\
&\quad=\ip{\frac{\delta H}{\delta t}\big(T(f)\big)}{\frac{\delta T}{\delta f}(f)h+E(h)}+\epsilon(h)\\
&\quad=\ip{\frac{\delta T}{\delta f}^{*}(f)\frac{\delta H}{\delta t}\Big(T(f)\Big)}{h}+\ip{\frac{\delta H}{\delta t}\big(T(f)\big)}{E(h)} +\epsilon(h),
\end{split}
\]
where the error-terms $E(h)=o(\norm{h})$, and $\epsilon(h) =o(\norm{T(f+h)-T(f)})$. 
Then, by Cauchy–Schwarz inequality we have $\ip{\frac{\delta H}{\delta t}\big(T(f)\big)}{E(h)}= o(\norm{h})$.
Thus, it remains to show that $o(\norm{T(f+h)-T(f)})$ implies $o(\norm{h})$. 
Since the variation $\frac{\delta T}{\delta f}$ exists, it follows that 
$T(f+h)-T(f) =\frac{\delta T}{\delta f}h+o(\norm{h})$. 
Now, since by definition $\frac{\delta T}{\delta f}$ is a bounded operator, we have
\[\norm{T(f+h)-T(f)} \leq \norm{\frac{\delta T}{\delta f}}\norm{h}+o(\norm{h}),\]
so $o(\norm{T(f+h)-T(f)})=o(\norm{h})$.
 
\end{proof}

\begin{corollary}\label{lemma cons}
Let $\cH$ be a real separable Hilbert space. Let $S:{\rm Dom}(S)\rightarrow\RR$ be a functional over the domain ${\rm Dom}(S)\subset\cH$. Suppose that ${\rm Dom}(S)\cup\{0\}$ is a dense linear subspace of $\cH$. Suppose that there exists a functional $H:{\rm Dom} (H)\rightarrow\RR$ such that $S(f) = H(\frac{f}{\norm{f}})$, where ${\rm Dom} (H)\cup\{0\}$ is a dense linear subspace of $\cH$. Then,  for every non-zero $f\in {\rm Dom }(H)$ such that the variation $\frac{\delta H}{\delta y}\Big(\frac{f}{\norm{f}}\Big)$ exists, 
the variation of $S$ exists at $f$ and is equal to 
\[\frac{\delta S}{\delta f}\quad= \frac{1}{\norm{f}}\frac{\delta H}{\delta y}\Big(\frac{f}{\norm{f}}\Big) - \ip{\frac{\delta H}{\delta y}\Big(\frac{f}{\norm{f}}\Big)}{\frac{f}{\norm{f}}}\frac{f}{\norm{f}^2}.\]
\end{corollary}
\begin{proof}
By Lemma \ref{lemma normalization var}, the variation $\frac{\delta \frac{ f}{\norm{f}}}{\delta f}$ exists for every non-zero $f$, and is a self-adjoint operator. Thus, by Lemma \ref{lemma chain rule functional-operator} 
\[\frac{\delta S}{\delta f}\quad= \frac{\delta \frac{ f}{\norm{f}}}{\delta f} \Bigg(\frac{\delta H}{\delta y}\Big(\frac{f}{\norm{f}}\Big)\Bigg).\]
Plugging in the formula (\ref{eq normalization var}) for $\frac{\delta \frac{ f}{\norm{f}}}{\delta f}$  yields the  result.  
\end{proof}

\subsection{Variations of Localization Measures}
\label{Variations of expected values and variances}

We now derive formulas for the variations localization measures based on observables. 

Throughout this section, $T$ denotes a symmetric, possibly unbounded, operator in a separable Hilbert space $\cH$. 
The domain of $T$, denoted by ${\rm Dom}(T)$, is assumed to be a dense linear subspace of $\cH$, and the variations are computed with respect to the realification $\cH^{\RR}$.
First, we compute the variation of the expected value of a (generally) non-normalized vector with respect to $T$.  

\begin{lemma}\label{lemma exp variance} 
Consider a Hilbert space $\cH$ and its realification $\cH^{\RR}$. Let $T$ be a self-adjoint operator on $\cH$, and consider the functional $f\mapsto e_{f}(T)$. 
Then, the variation $\frac{\delta e_f(T)}{\delta f}$ with respect to $\cH^{\RR}$ exists in the domain of $T$, and satisfies 
\begin{equation}
\frac{\delta e_{f}(T)}{\delta f} = 2T f.
\end{equation}
\end{lemma}
\begin{proof}
Compute, 
\[e_{f+h}(T) = \ip{T(f+h)}{ f+h} =  e_{f}(T) + 2 \Re \ip{Tf}{ h} + o(\norm{h}),\]
where the $o$ notation is with respect to ${\rm Dom}(T)\ni h \rightarrow 0$. 
\end{proof}

Next, we compute the variation of the variance with respect to (generally) non-nocrmalized vectors. For that, for an operator $T$ with domain ${\rm Dom}(T)$ in the Hilbert space $\cH$, the operator $T^2$ is defined as the composition of $T$ with itself over the domain \[{\rm Dom}(T^2)=\{f\in {\rm Dom}(T) \ |\ Tf\in {\rm Dom}(T)\}.\]

\begin{lemma}\label{Lemma uncons var}
Consider a Hilbert space $\cH$ and its realification $\cH^{\RR}$. Let $T$ be a self-adjoint operator on $\cH$, and consider the functional $f\mapsto v_{f}(T)$ defined over ${\rm Dom}(T)$. 
Then, the variation $\frac{\delta v_f(T)}{\delta f}$ with respect to $\cH^{\RR}$ exists for every $f\in {\rm Dom} (T^2)$, and satisfies 
\[\frac{\delta v_{f}(T)}{\delta f} = 2(T-e_{f}(T))^2 f.\]
\end{lemma}

\begin{proof}
By Definition \ref{var} and the fact $T$ is symmetric, for $f\in {\rm Dom} (T^2)$ 
\[ v_{f}(T) =\ip{(T-e_{f}(T))^2{f}}{f}.\]  
Substituting $f\mapsto f +h$ we get,
\[
\begin{split}
v_{f+h}(T) &=\ip{\big(T-e_{f+h}(T)\big)^2 (f+h)}{f+h}\\
&=\ip{\big(T-e_{f+h}(T)\big)^2f}{f}\\
&\quad+ 2{\rm Re}\ip{\big(T-e_{f+h}(T)\big)^2 f}{h} + o(\norm{h}^2)\\
&=\ip{\Big(T-e_{f}(T)  -2\ip{Tf}{h}_{\cH^{\RR}}\Big)^2f}{f}\\
&\quad+ 2{\rm Re}\ip{\Big(T-e_{f}(T)-2\ip{Tf}{h}_{\cH^{\RR}}\Big)^2 f}{h} + o(\norm{h}^2)\\
&=v_f(T) + 4\ip{Tf}{h}_{\cH^{\RR}}^2 + 2{\rm Re}\ip{\Big(T-e_{f}(T)\Big)^2 f}{h} + o(\norm{h}^2)\\ 
& = v_f(T) + \ip{2\Big(T-e_{f}(T)\Big)^2 f}{h}_{\cH^{\RR}} + o(\norm{h}) . 
\end{split}
\]
 
\end{proof}
From Corollary \ref{lemma cons}, and Lemmas \ref{lemma exp variance}, and \ref{Lemma uncons var},  we directly deduce the following.
\begin{proposition}\label{prop uncons exp and var}
Let $\cH$ be a Hilbert space with realification $\cH^{\RR}$. Let $T$ be a self-adjoint operator in $\cH$, and consider the functionals $f\mapsto e_{\frac{f}{\norm{f}}}(T)$ and $f\mapsto v_{\frac{f}{\norm{f}}}(T)$ defined on ${\rm Dom}(T)\setminus\{0\}$. 
\begin{enumerate}
    \item \label{exp var}
    The variation of  $e_{\frac{f}{\norm{f}}}(T)$ with respect to $\cH^{\RR}$ exists for every $f\in{\rm Dom}(T)\setminus\{0\}$ and
    \[ 
\begin{split}
\frac{\delta e_{\frac{f}{\norm{f}}}(T)}{\delta f} 
&\quad=2T\frac{f}{\norm{f}^2} - 2e_{\frac{f}{\norm{f}}}(T)\frac{f}{\norm{f}^2}
\end{split}
\]
    \item\label{constrained variance var}
    The variation of  $v_{\frac{f}{\norm{f}}}(T)$ with respect to $\cH^{\RR}$ exists for every $f\in{\rm Dom}(T^2)\setminus\{0\}$ and
    \[
\frac{\delta v_{\frac{f}{\norm{f}}}(T)}{\delta f} 
\quad= 2(T-e_{\frac{f}{\norm{f}}}(T))^2\frac{f}{\norm{f}^2} - 2v_{\frac{f}{\norm{f}}}(T)\frac{f}{\norm{f}^2}.
\]
\end{enumerate}
\end{proposition}

Last, the following lemma is proved similarly to Lemma \ref{Lemma uncons var}. 
\begin{lemma} \label{lemma norm op}
Let $\cH$ be a separable Hilbert space with realification $\cH^{\RR}$, and $T$ a self-adjoint operator in $\cH$ over the domain ${\rm Dom}(T)$. The variation of $\norm{Tf}^2$ with respect to $\cH^{\RR}$ exists at every $f\in {\rm Dom}(T^2)$, and satisfies
\begin{equation}\label{eq norm var}
\frac{\delta \norm{Tf}^2}{\delta f} =  2T^2f. 
\end{equation}
\end{lemma}

\section{Optimizing the Uncertainty Minimization}
\label{sec minimization}

In this section we use the  calculus of variation results from Section \ref{calculus of variations}   for finding a local minimum of the Minimization Problem \ref{Min_prob1}.

\subsection{Variations in Windows and Signals}
\label{window vs signal variation}

Some of the terms in the uncertainty $\cL(\hf)$ are defined with the inner product of $\cS$, and some with that of $\cW$. 
It is thus useful to understand the relation between the variation of a functional $S$ with respect to the realificated signal space $\cS^{\RR}$, and its variation with respect to  $\cW^{\RR}$ (see Definition \ref{DefReal}). We denote the variations of $S$ with respect to $\cS^{\RR}$ and $\cW^{\RR}$ by
$\frac{\delta^{\cS} S}{\delta^{\cS} \hf}$ and $\frac{\delta^{\cW} S}{\delta^{\cW} \hf}$ respectively.

\begin{lemma} \label{var as a signal}
Let $S:{\rm Dom} (S)\rightarrow \CC$ be a  functional, where  ${\rm Dom} (S)\cup \{0\}$ is a dense linear subspace of $\cW\cap\cS$. Let $\hf\in {\rm Dom} (S)$. Then, if the variation $\frac{\delta^{\cS} S}{\delta^{\cS} \hf}$ exists at $f$, and satisfies $\w\frac{\delta^{\cS} S}{\delta^{\cS} \hf}\in \cW$, then $\frac{\delta^{\cW} S}{\delta^{\cW} \hf}$ exists at $f$ and 
\[
        \frac{\delta^{\cW} S}{\delta^{\cW} \hf} = \w\frac{\delta^{\cS} S}{\delta^{\cS} \hf}.
\]
\end{lemma}
\begin{proof}
From Proposition \ref{prop var existance}, $\frac{\delta^{\cS} S}{\delta^{\cS} \hf}$ exists 
if and only if 
\begin{equation} \label{eq var by S}
S(\hf + \hh) - S(\hf) = {\rm Re}\ip{\frac{\delta^{\cS} S}{\delta^{\cS} \hf}}{\hh}_{\cS}+o(\norm{\hh}_{\cS}),     
\end{equation}
where the $o$ notation is with respect to ${\rm Dom}(S)\ni h \rightarrow 0$. 
Note that for a function $\hg \in\cW\cap\cS$, such that $\w\hg$ is in $\cW$, it holds
\[
\ip{\hg}{\hh}_{\cS^{\RR}} = {\rm Re}\Big(\int_{\RR} \hg(\w)\overline{\hh(\w)} d\w\Big) = \ip{\w \hg}{\hh}_{\cW^{\RR}}.    
\]
Thus, (\ref{eq var by S}) is equivalent to 
\begin{equation}\label{eq var by W}
    S(\hf + \hh) - S(\hf)
    =  {\rm Re}\ip{\w \frac{\delta^{\cS} S}{\delta^{\cS} \hf}}{\hh}_{\cW}+o(\norm{\hh}_{\cS}), 
\end{equation}
which, by Proposition, \ref{prop var existance} is equivalent to the existence of $\frac{\delta^{\cW} S}{\delta^{\cW} \hf}$ and its equality to $\w \frac{\delta^{\cS} S}{\delta^{\cS} \hf}$.   Here, note that $o(\norm{\hh}_{\cS})=o(\norm{\hh}_{\cW})$ for any $h\in\cW\cap\cS$, since the denominator in the o-notation of Definition \ref{def_o-not}.(2) satisfies $\norm{th}_{\cW}\geq \norm{th}_{\cS}\frac{\norm{h}_{\cW}}{\norm{h}_{\cS}}$.
    
 \end{proof}
 The above lemma can be viewed as a formula for computing the variation of the functional $S$ with respect to the window $\hf$, given a formula for its variation with respect to $\hf$ as a signal.

\subsection{The Unconstrained Variation of the Uncertainty}
\label{The unconstrained variation}
In this section, we compute the variations of the functional $\cL$. We present the results under the assumptions that $e_{\hf}(\Tx)=e_{\hf}(\ln(\abs{\w}))=0$, where in the mapping $f\mapsto f+h$ underlying the definition of the variation, $f+h$ is not restricted to lie in the constraint  $e_{\hf}(\Tx)=e_{\hf}(\ln(\abs{\w}))=0$.

We first define a subset of ${\rm Dom}_{\cW\cap\cS}(\cL)$ in which the variation of $\cL$ exists. This is done by combining the requirements from Proposition \ref{prop uncons exp and var} and \ref{var as a signal} applied on the different components of (\ref{cL}).
\begin{definition}
The domain ${\rm Dom}^{\delta}_{\cW\cap\cS}(\cL)\subset \cW\cap\cS$ is defined as the set of all $0\neq \hf\in L^2(\RR)$ such that
\begin{enumerate}
    \item 
    $\hf$ is differentiable in every compact interval with absolutely continuous derivative, and
\[ \frac{\partial^2}{\partial\w^2}\hf\in L^2(\RR)  , \quad \sqrt{\abs{\w}}\frac{\partial^2}{\partial\w^2}\hf\in L^2(\RR) {\rm \quad and} \quad  \Big(\sqrt{\abs{\w}^3}\frac{\partial^2}{\partial\w^2} + \sqrt{\abs{w}}\frac{\partial}{\partial\w}\Big)\hf\in L^2(\RR)\]
\item
$\w\mapsto \frac{1}{\w^2}\hf(\w) \in L^2(\RR)$
\item
$\w\mapsto \sqrt{\abs{\w}}\ln(\abs{\w})^2\hf(\w) \in L^2(\RR)$.
\end{enumerate}
\end{definition}

 \begin{theorem}\label{thm uncos var L}
The variation of the 
functional $\cL$ with respect to $\cW^{\RR}$ exists at every $\hf\in {\rm Dom}^{\delta}_{\cW\cap\cS}(\cL)$, and
for $\hf$ satisfying $e_{\hf}(\Tx)=e_{\hf}\big(\ln(\abs{\w})\big)=0$,
is given by
\[
    \frac{\delta^{\cW} \cL}{\delta^{\cW} \hf} =\frac{\delta^{\cW} }{\delta^{\cW}\hf}\bigg(v_{\ambig}(A)\bigg)+\frac{\delta^{\cW} }{\delta^{\cW}\hf}\bigg(v_{\ambig}(B)\bigg),\]
    where, 
    \begin{equation}
    \begin{split}
    & \frac{\delta^{\cW} }{\delta^{\cW}\hf}\bigg(v_{\ambig}(A)\bigg)(\hf) \\
    & =
    2\w\bigg(\ln(\abs{\w})^2-v^{\cS}_{\frac{\hf}{\norm{\hf}_{\cS}}}\big(\ln(\abs{\w})\big)\bigg)\frac{\hf}{\norm{\hf}_{\cS}^2} \\
    &\quad+ 2\Bigg(\Big(\ln(\abs{\w})-e_{\frac{\hf}{\norm{\hf}_{\cW}}}^{\cW}\big(\ln(\abs{\w})\big)\Big)^2 - v_{\frac{\hf}{\norm{\hf}_{\cW}}}^{\cW}\big(\ln(\abs{\w})\big)\Bigg)\frac{\hf}{\norm{\hf}_{\cW}^2}\\
    \end{split}
    \end{equation}
    and
    \begin{equation}
    \begin{split}
    & \frac{\delta^{\cW} }{\delta^{\cW}\hf}\bigg(v_{\ambig}(B)\bigg)(\hf) \\
    & = 
     2\w\left(-\frac{\partial^2}{\partial\w^2}-v^{\cS}_{\frac{\hf}{\norm{\hf}_{\cS}}}\Big(i\frac{\partial}{\partial\w}\Big)\right)\frac{\hf}{\norm{\hf}_{\cS}^2} \\
      &\quad+ 2\Bigg(-\w^2\frac{\partial^2}{\partial\w^2} -\w\frac{\partial}{\partial\w} -v_{\frac{\hf}{\norm{\hf}_{\cW}}}^{\cW}\Big(i\w\frac{\partial}{\partial\w}\Big)\Bigg)
        \frac{\norm{\frac{1}{w}\hf}_{\cS}^2}{\norm{\hf}^2_{\cS}}   \frac{\hf}{\norm{\hf}_{\cW}^2}\\
    &\quad+
2v_{\frac{\hf}{\norm{\hf}_{\cW}}}^{\cW}\Big(i\w\frac{\partial}{\partial\w}\Big)
\left(\frac{1}{\w}\frac{\hf}{\norm{\hf}_{\cS}^2} - \w\frac{\norm{\frac{1}{w}\hf}_{\cS}^2}{\norm{\hf}^4_{\cS}}\hf \right).
\end{split}    
\end{equation}
\begin{proof}
The formulas are proven directly by applying Proposition \ref{prop uncons exp and var}, Lemma \ref{lemma norm op}, Lemma \ref{var as a signal}, the product rule, and the fact that $e^{\cW}_{\frac{\hf}{\norm{\hf}_{\cW}}}(i\w\frac{\partial}{\partial\w})=0$, which follows from $e^{\cS}_{\frac{\hf}{\norm{\hf}_{\cS}}}(i\frac{\partial}{\partial\w})=0$.   
\end{proof}
\end{theorem}

\subsection{Variations Under Expected Value Constraints}
\label{Adding the constraints0}

Our next goal is to add the constraints $e^{\cS}_{\hf}(\ln(\abs{\w}))= e^{\cS}_{\hf}(i\frac{\partial}{\partial\w})=0$ to the variational analysis.
In the previous section, we calculated the unconstrained variations of $\cL$. Informally, adding the constraints to the variation involves 
formulating the Lagrange multipliers, namely, projecting the unconstrained variation $\frac{\delta S}{\delta \hf}$ to the tangent space of the constraint surface. Intuitively,  taking an infinitesimal step along the constrained variation preserves the constraints.

The general Lagrange multipliers formulation is given by the following definition and theorem.

\begin{definition}
Let $S$, $c_1$ and $c_2$ be densely defined functionals in the Hilbert space $\cH$. Let $f$ be in the domains of $S$, $c_1$ and $c_2$, and suppose that the variations 
of $S$, $c_1$ and $c_2$ exist at $f$, and that $c_1(f)=c_2(f)=0$.
The \emph{constrained variation} of $S$ at $f$, under the constraints $c_1(f)=c_2(f)=0$, is defined to be
\begin{equation}\label{eq def cons}
\frac{\delta S}{\delta f}^{\emph{cons}}:= 
\frac{\delta S}{\delta f} - \lambda \frac{\delta c_1}{\delta f} -\mu \frac{\delta c_2}{\delta f},
\end{equation}    
where $\lambda$ and $\mu$, called the \emph{Lagrange multipliers} associated to the constraints, are scalars that satisfy
\begin{equation}
\label{eq_lag}
\ip{\frac{\delta S}{\delta f}^{\emph{cons}}}{\frac{\delta c_1}{\delta f}}_{\cH}=0, \quad \ip{\frac{\delta S}{\delta f}^{\emph{cons}}}{\frac{\delta c_2}{\delta f}}_{\cH}=0.
\end{equation}
\end{definition}

In the above definition, if $\frac{\delta c_1}{\delta f}$ and $\frac{\delta c_2}{\delta f}$ are not co-linear, equation (\ref{eq_lag}) has a unique solution. If $\frac{\delta c_1}{\delta f}$ and $\frac{\delta c_2}{\delta f}$ are co-linear, then $\l$ and $\mu$ are not unique, but $\frac{\delta S}{\delta f}^{\emph{cons}}$ is. Also, note that $\l$ and $\mu$ depend on $f$, as $\frac{\delta S}{\delta f}$, $\frac{\delta c_1}{\delta f}$ and $\frac{\delta c_2}{\delta f}$ do.
One motivation for the above definition comes from the following extreme value theorem, which is restricted to expected value constraints. For the theorem, given self-adjoint operators $T_1$ and $T_2$ with domains ${\rm Dom}(T_1)$ and ${\rm Dom}(T_2)$ respectively, we denote by ${\rm Dom}(T_1T_2)$ the set of all $f\in {\rm Dom}(T_2)$ such that $T_2f\in{\rm Dom}(T_1)$.

\begin{theorem}[Fermat's Extreme Value Theorem Under Constraints]
\label{fermat2}
Let $S$ be a densely defined functional in the real Hilbert space $\cH$. Let $T_1$ and $T_2$ be self-adjoint operators in $\cH$, and consider the expected value constraints 
$e_f(T_1)=e_f(T_2)=0$.
Denote the domains of $S$, $T_1$ and $T_2$ respectively by ${\rm Dom}(S)$, ${\rm Dom}(T_1)$ and ${\rm Dom}(T_2)$.
 Let $f_0\in {\rm Dom}(S)\cap {\rm Dom}(T_1^2)\cap {\rm Dom}(T_1 T_2)\cap {\rm Dom}(T_2^2)$ and suppose that the variation of $S$ exists at $f_0$, $e_{f_0}(T_1)=e_{f_0}(T_2)=0$, and $T_1f_0$ and $T_2f_0$ are not co-linear.
 
 If $S$ attains a local minimum in the domain $e_f(T_1)=e_f(T_2)=0$ at $f_0$, then there exist scalars $\lambda$ and $\mu$ such that 
 \begin{equation}
 \label{eq_LagMult}
     \frac{\delta S}{\delta f}(f_0) - \lambda T_1 f_0 -\mu T_2 f_0 = 0.
 \end{equation}
\end{theorem}

It is trivial to see that $\l$ and $\mu$ in the above theorem are Lagrange multipliers (satisfying (\ref{eq_lag})).
In the following, we offer a proof of the one constraint counterpart of Theorem \ref{fermat2}. Namely, under the constraint $e_f(T)=0$, we show that any local minimum point of $S$ at $f$, where $f\in{\rm Dom}(S)\cap{\rm Dom}(T^2)$, $Tf\neq 0$, and where $\frac{\partial S}{\partial f}(f)$ exists, must satisfy
\begin{equation}
\label{eq_temp3}
   \frac{\delta S}{\delta f}(f) - \lambda T f = 0 
\end{equation}
for some $\lambda$. The general case is shown similarly.

\begin{proof}
We denote in short the expected value functional by $e:f\mapsto e_f(T)$, and note that $\frac{\delta}{\delta f}e(f)$ exists and is equal to $2Tf$ by Lemma \ref{lemma exp variance}.

Let $f$ be as in the theorem.
Suppose that there is no $\l$ that satisfies (\ref{eq_temp3}). In this proof, we build a new point $q\neq f$ on the constraint $e_q(T)=0$ for which $S(q)<S(f)$, which contradicts the assumption.
In the following, we construct such a $q$ of the form
\begin{equation}
\label{eq_temp4}
q=f-\e\frac{\delta }{\delta f}S +\big(\l(\e)+k\big)2Tf,
\end{equation}
with
\[
    \l(\e) = \frac{\e \ip{\frac{\delta }{\delta f}S}{2Tf}}{\norm{2Tf}^2} 
\]
for small enough nonzero $\e$ and $k=o(\e)$. 
Here, indeed $q\neq f$, since
(\ref{eq_temp4}) has $q-f$ of the form $\e\big(\frac{\delta S}{\delta f} + \l' \frac{\delta e}{\delta f}\big)$, which must be nonzero by assumption.

First, we show the existence of $k$ such that $q$ is on the constraint $e(q)=0$.
\begin{equation}
\label{eq_temp5}
\begin{split}
   e(q) = & \ip{T\Big(f-\e\frac{\delta }{\delta f}S +\big(\l(\e)+k\big)2Tf \Big)}{f-\e\frac{\delta }{\delta f}S +\big(\l(\e)+k\big)2Tf }\\
   = & e(f) + 2\ip{Tf}{-\e\frac{\delta }{\delta f}S +\big(\l(\e)+k\big)2Tf } \\
    & +\ip{T\Big(-\e\frac{\delta }{\delta f}S +\big(\l(\e)+k\big)2Tf \Big)}{-\e\frac{\delta }{\delta f}S +\big(\l(\e)+k\big)2Tf }.
\end{split}
\end{equation}
Let us compute the different terms in the right-hand-side of (\ref{eq_temp5}).
\begin{equation}
    \label{eq_temp7}
    \begin{split}
     & 2\ip{Tf}{-\e\frac{\delta }{\delta f}S +\big(\l(\e)+k\big)2Tf } \\
     & \quad=2\ip{Tf}{-\e\frac{\delta }{\delta f}S +\frac{\e \ip{\frac{\delta }{\delta f}S}{2Tf}}{\norm{2Tf}^2}2Tf }  +2\ip{Tf}{k2Tf }\\
     & \quad= 4k\ip{Tf}{Tf }
    \end{split}
\end{equation}
Now, the second term of (\ref{eq_temp5}) satisfies
\begin{equation}
    \label{eq_temp8}
    \begin{split}
      &  \ip{T\Big(-\e\frac{\delta }{\delta f}S +\big(\l(\e)+k\big)2Tf \Big)}{-\e\frac{\delta }{\delta f}S +\big(\l(\e)+k\big)2Tf } \\
      & \quad =\ip{T\Big(-\e\frac{\delta }{\delta f}S +\l(\e)2Tf \Big)}{-\e\frac{\delta }{\delta f}S +\l(\e)2Tf } \\
      &  \quad \quad +2\ip{T\Big(-\e\frac{\delta }{\delta f}S +\l(\e)2Tf \Big)}{ 2kTf }+\ip{2kT^2f }{2kTf }.
    \end{split}
\end{equation}
Thus, combining (\ref{eq_temp7}) with (\ref{eq_temp8}), we can write
\begin{equation}
    \label{eq_temp9}
    e(q)=ak^2 +(b+\e c) k +d\e^2.
\end{equation}
where $a,b$ and $c$ are  constants that do not depend on $\e$. By the fact that 
 $f$ is not in the kernel of $T$, we must have $b\neq 0$. 
For small enough $\e$, there is hence a solution $k$ to the equation $ e(q)=ak^2 +(b+\e c) k +d\e^2=0$ with $k=O(\e^2)$.
Hence $q$ satisfies the constraint $e(q)=0$ for any small enough $\e$ with a corresponding choice of $k=o(\e)$.

We next show that $S(q)<S(f)$ for some choice of $\e$. For small enough $\e$, the $k$ term in the definitoin (\ref{eq_temp4}) of $q$ cannot cancel the $\e$ and $\l(\e)$ terms, since $k=o(\e)$. We hence we get, by PRoposition \ref{prop var existance},
\[S(q)= S(f) + \ip{\frac{\delta}{\delta f}S}{-\e\frac{\delta }{\delta f}S +\l(\e)2Tf} + \ip{\frac{\delta}{\delta f}S}{k2Tf} + o(\e), \]
where first term does not vanish by the non-colinearity assumption of $Tf$ and $\frac{\delta}{\delta f}S(f)$, and the second and third terms can be made small enough relative to the first term by choosing small enough $\e$. Hence, there exists a small as we wish $\e$ such that $S(q)<S(f)$ and $e(q=0)$, so $f$ is not a local minimum.

\end{proof}

Similarly to the above proof, we can also show the following. Given a point on the constraint, by taking a small enough step in the direction of $\frac{\delta S}{\delta f}^{\emph{cons}}$, and projecting the result to the constraint $e_{f}(T_1)=e_f(T_2)=0$, we can decrease the value of the functional $S$, unless $\frac{\delta S}{\delta f}^{\emph{cons}}=0$. This is the basis of the gradient descent approach for finding local minima.

\subsection{The Variation of the Uncertainty Under the Expected Value Constraints}
\label{Adding the constraints}

Next, we give an explicit formula for $\frac{\delta \cL}{\delta \hf}^{\emph{cons}}$. This is a direct result of Lemma \ref{lemma exp variance} and (\ref{eq_lag}).
\begin{proposition}
The constrained variation
of $\cL$ with respect to $\cW^{\RR}$ 
exists at every $\hf\in {\rm Dom}^{\delta}_{\cW\cap\cS}(\cL)$, and satisfies
\begin{equation}\label{eq constrained}
\frac{\delta^{\cW} \cL}{\delta^{\cW} \hf}^{\emph{cons}}=
\frac{\delta^{\cW} \cL}{\delta^{\cW} \hf} -  \lambda \ln(\abs{\w})\hf - \mu i\frac{\partial}{\partial\w}\hf,
\end{equation}
where the Lagrange multipliers $\mu$ and $\lambda$ satisfy
\begin{equation}\label{eq lagrange mul}
\begin{split}
& \begin{pmatrix}
\lambda\\
\mu
\end{pmatrix} \\
& =
\begin{pmatrix}
\norm{\ln(\abs{\w})\hf}^2_{\cW}& {\rm Re}\ip{\Tx\hf}{\ln(\abs{\w})\hf}_{\cW}\\
{\rm Re}\ip{\ln(\abs{\w})\hf}{\Tx\hf}_{\cW}&\norm{\Tx\hf}^2_{\cW}
\end{pmatrix}^{-1}
\begin{pmatrix}
{\rm Re}\ip{\frac{\delta^{\cW}\cL}{\delta^{\cW}\hf}}{\ln(\abs{\w})\hf}_{\cW}\\
{\rm Re}\ip{\frac{\delta^{\cW}\cL}{\delta^{\cW}\hf}}{\Tx\hf}_{\cW}
\end{pmatrix}
\end{split}
\end{equation}
\end{proposition}

In the special case, where $\hf$ is real valued,  formula (\ref{eq lagrange mul}) can be simplified. 
\begin{corollary}\label{prop real valued constrained} The constrained variation
of $\cL$ with respect to $\cW^{\RR}$ at a real valued $\hf\in {\rm Dom}^{\delta}_{\cW\cap\cS}(\cL)$
is given by
\begin{equation}\label{eq L cons var real}
\frac{\delta^{\cW} \cL}{\delta^{\cW} \hf}^{\emph{cons}}=
\frac{\delta^{\cW} \cL}{\delta^{\cW} \hf} - \frac{{\rm Re}\ip{\frac{\delta^{\cW} \cL}{\delta^{\cW} \hf}}{\ln(\abs{\w})\hf}_{\cW}}{\norm{\ln(\abs{\w}) \hf}_{\cW}^2}\ln(\abs{\w})\hf.
\end{equation}
In particular, $\frac{\delta^{\cW}\cL}{\delta^{\cW} \hf}^{\emph{cons}}$ is real valued as well. 
\end{corollary}
\begin{proof}
By Theorem \ref{thm uncos var L}, for a real valued $\hf$, both $\frac{\delta\cL}{\delta\hf}$ and $\ln(\abs{\w})\hf$ are real valued. In addition, $\Tx\hf$ takes values in $i\RR$. 
Thus, 
\[{\rm Re}\ip{\Tx\hf}{\ln(\abs{\w})\hf}_{\cW}={\rm Re}\ip{\frac{\delta\cL}{\delta\hf}}{\Tx\hf}_{\cW}=0.\]
As a result, by (\ref{eq lagrange mul}),
$\mu = 0$, and $\l=\frac{{\rm Re}\ip{\frac{\delta^{\cW} \cL}{\delta^{\cW} \hf}}{\ln(\abs{\w})\hf}_{\cW}}{\norm{\ln(\abs{\w}) \hf}_{\cW}^2}$.   
\end{proof}

\section{Numerical Results}\label{sec implementation}

We illustrate the uncertainty minimization problem numerically by implementing a discrete gradient descent algorithm. The scheme is based on discretizing the frequency line on a grid, and  implementating a discrete version of the constrained variation of Corollary \ref{prop real valued constrained}.
 The derivatives are discretized via central difference.  
We initialize the window as a Gaussian $\hat{f}_0$, translated in the $y$ axis to satisfy $\hat{f}_0(0)=0$, and zeroed out for negative values. The initial $f_0$ is chosen with expected values of time and scale equal to 0. We call such an $f_0$ a \emph{truncated Gaussian}.

 We choose the variance of the initial Gaussian $\hat{f}_0$ optimally -- to minimize the uncertainty over the family of truncated Gaussians. This means that the optimization process only changes the shape of the Gaussian, and not its first and second moments. In Figure, \ref{fig:initial_window} we show the initial condition, and in Figure \ref{fig:final_window}, the optimal window. The uncertainty 11.427 of the initial condition improves to 8.0619 in the optimal window (see Figure \ref{fig:uncertainty_window} ). Note that the shape of the optimal window is similar to the shape of the initial condition. Hence, from a signal-processing/feature-extraction point of view, both windows are reasonable and roughly measure the same features. However, the shape of the uncertainty minimizing window is optimized for best phase space localization.

\begin{figure*}
\centering
  \includegraphics[width=0.71\textwidth]{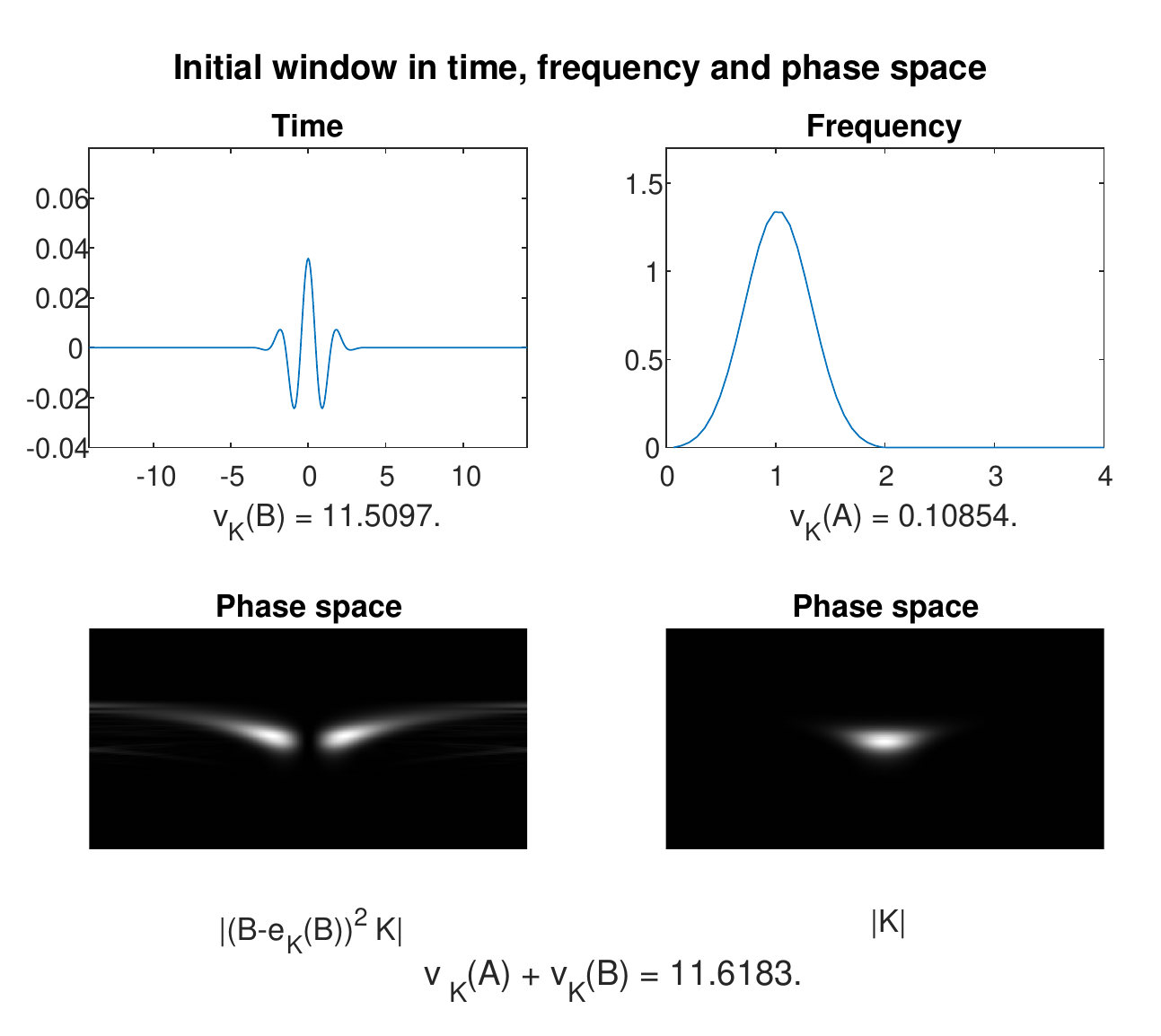}
\caption{\textbf{Initial window} in time, frequency, and phase space. To make the localization more visible in phase space, we show  $|(B-e_K(B)^2K(a,b,c)|$ in addition to  $|K(a,b,c)|$
, where $K:=V_f(f)$ is the ambiguity function.} 
\label{fig:initial_window}     
\end{figure*}
\begin{figure*}
\centering
  \includegraphics[width=0.71\textwidth]{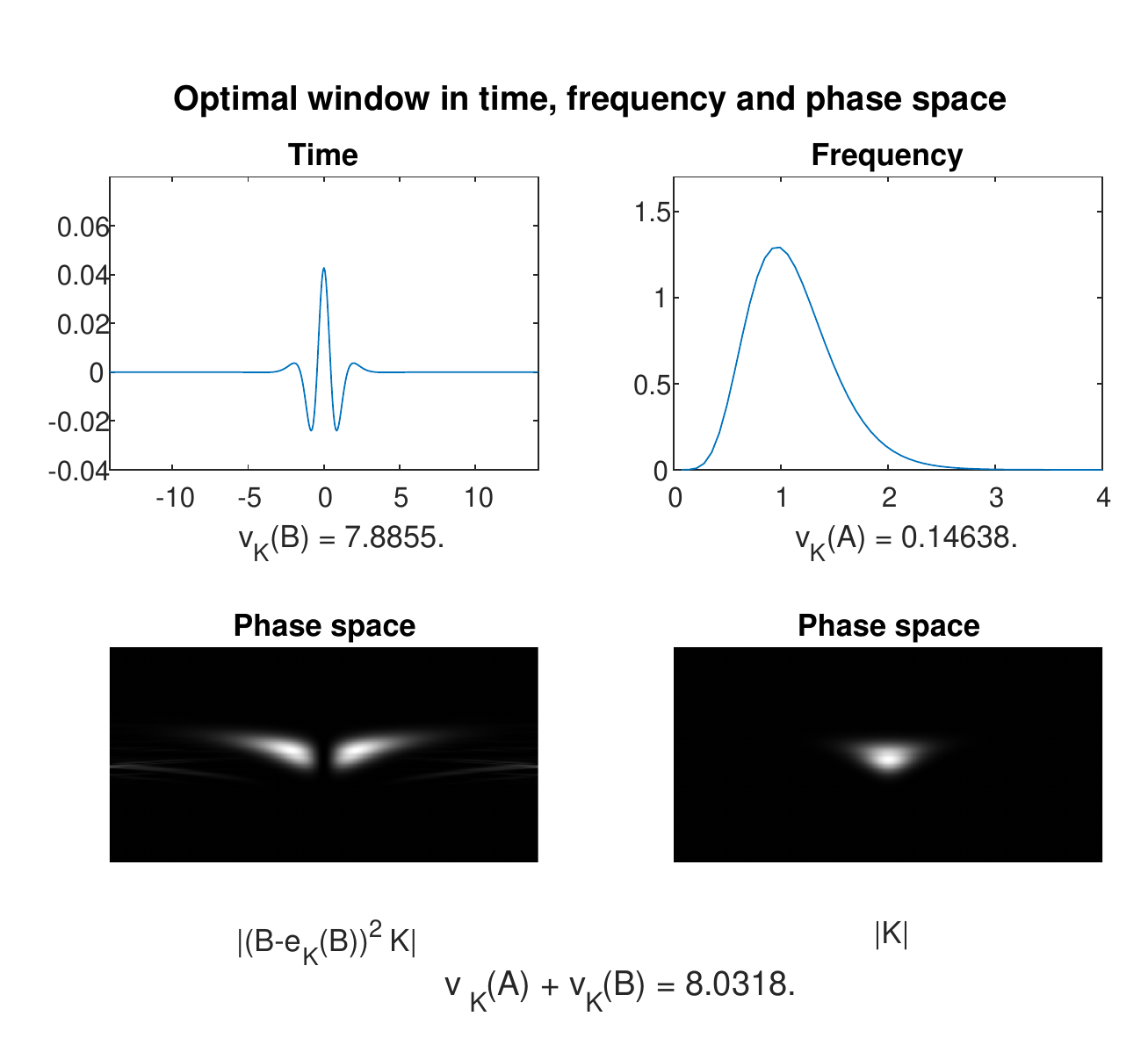} 
\caption{\textbf{Final window} in time, frequency, and phase space. To make the localization more visible in phase space, we show  $|(B-e_K(B)^2K(a,b,c)|$ in addition to  $|K(a,b,c)|$
, where $K:=V_f(f)$ is the ambiguity function.}
\label{fig:final_window}       
\end{figure*}

\begin{figure*}
\centering
  \includegraphics[width=0.45\textwidth]{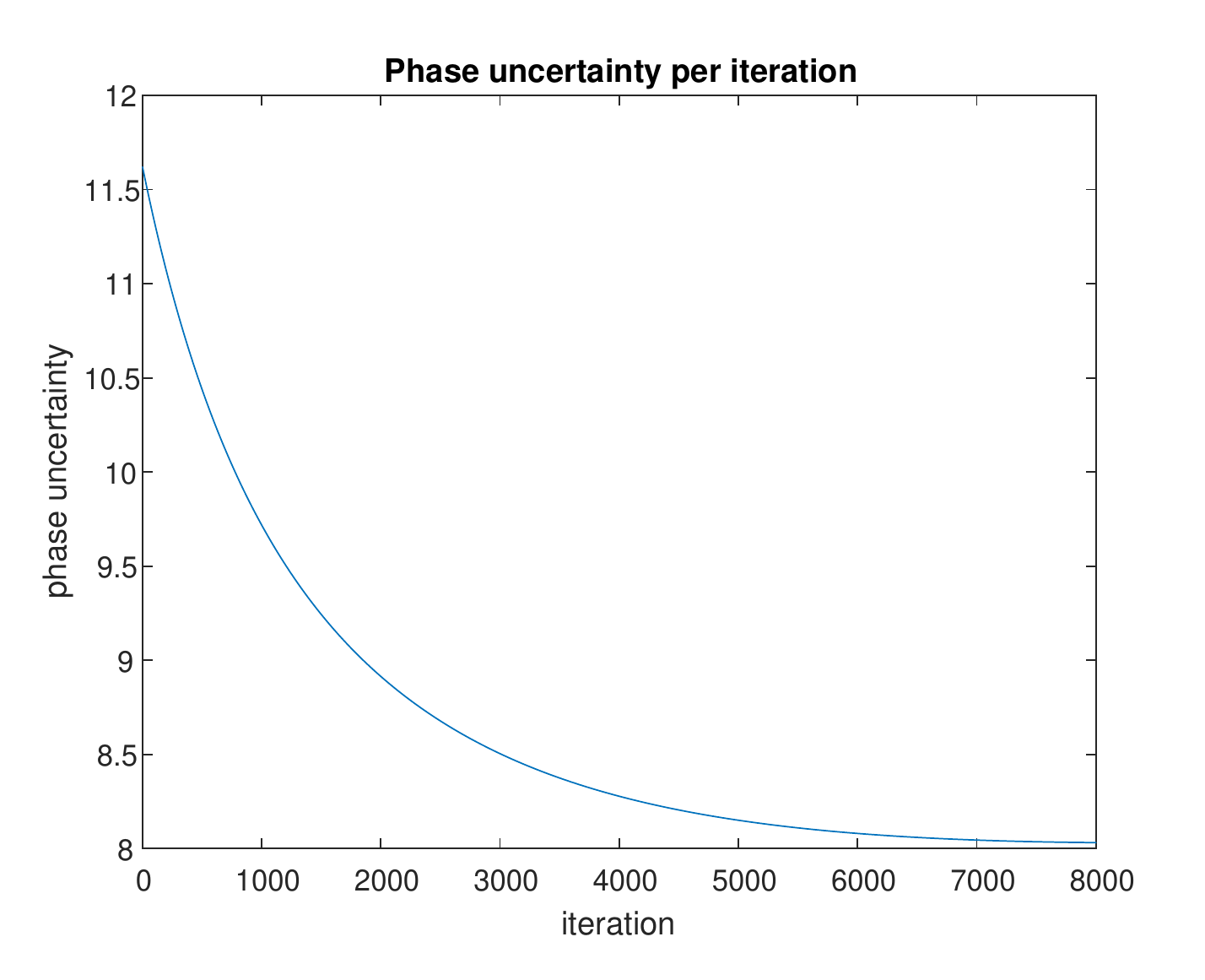}
\caption{Phase space uncertainty per gradient descent  iteration.}
\label{fig:uncertainty_window}       %
\end{figure*}

\section{Summary}\label{sec summary}
In this work, we approached the task of constructing windows for the 1D wavelet transform from the uncertainty minimization point of view. We based the approach on localization via observables, as in \cite{Levie2020}. In contrast to \cite{Levie2020}, we defined an observable-based uncertainty  directly in phase space. Our proposed optimal windows promote sparsity in phase space in a non-asymptotic manner. 
Basing the computations on the wavelet-Plancherel theorem enabled us to formulate the 2D phase space variances as a combination of 1D signal and window localization measures. This allowed us to optimize the 1D window directly, as opposed to optimizing a general 2D function in phase space, which would require complicated constraints for restricting the function to be an ambiguity function.
While we studied the 1D wavelet transform in this paper, our technique can be seen as a step-by-step guide for computing optimal windows for every generalized wavelet transform based on a semi-direct product of physical quantities (see  \cite[Section 3.3]{Levie2020}), like, for example, the Shearlet transform.
\bibliographystyle{plain}
\bibliography{references}   

\begin{appendix}

\section{Direct Computation of the Pull-Backs of Phase Space Observables}\label{appendixA}

In this appendix we compute directly the pull-back formulas of the phase space observables for simple vectors $\hf\otimes\hs$. 

	\subsection{Pull-Back of the Scale Window-Signal Observable}
	\label{A:scale}
	In this section we proof the formula for $\bT_{a}(\hf\otimes\hs)$ given in Corollary \ref{cor_simp_obs} in a direct manner. To that end, we formulate $aV_f(s)(a,b,c)$ as a linear combination of simple wavelet functions.
	From (\ref{eq tpi}), we can deduce the wavelet transform in scale space
    \begin{equation}\label{tVfs}
    \begin{split}
	       & V_{\tf}(\ts)(a,b,c) = \ip{\tpi(a,b,c)\tf}{\ts}\\
	       &= \int_{-\infty}^{\infty}\overline{e^{-ib e^{-\sigma}}}\Big(\overline{\tf\big(\sigma-a, sign(c)\big)}\ts(\sigma,1)
	       + \overline{\tf\big(\sigma-a, sign(-c)\big)}\ts(\sigma,-1)\Big)d\sigma\\
	 \end{split} 
    \end{equation}
    From (\ref{tVfs}) we can compute  
    \[\begin{split}
    & aV_{\tf}(\ts)(a,b,c) \\
    & =  \int_{-\infty}^{\infty} a \overline{e^{-ib e^{-\sigma}}}\Big(\overline{\tf\big(\sigma-a, sign(c)\big)}\ts(\sigma, 1)
	       + \overline{\tf\big(\sigma-a, sign(-c)\big)}\ts(\sigma, -1)\Big)d\sigma,
	       \end{split}\]
	 and,
	\[\begin{split}
	& V_{\sigma \tf}(\ts)(a,b,c) \\
	& = \int_{-\infty}^{\infty}\overline{e^{-ib e^{-\sigma}}}(\sigma-a)\Big(\overline{\tf\big(\sigma-a, sign(c)\big)}\ts(\sigma, 1)
	       + \overline{\tf\big(\sigma-a, sign(-c)\big)}\ts(\sigma, -1)\Big)d\sigma.
	       \end{split}\]
	Thus,
	\[
	 aV_{\tf}(\ts)(a,b,c) = V_{\tf}(\sigma \ts)(a,b,c) - V_{\sigma \tf}(\ts)(a,b,c).   
	\]
	From the definition of the scale parameter $\sigma$ we get, 
	\[\bT_a (\hf\otimes \hs)=  -\hf \otimes \big(\ln(|\w|)\hs\big) \ \ +\ \ \big(\ln(|\w'|) \hf\big) \otimes \hs.\]
	
	\subsection{Pull-Back of the Time Window-Signal Observable}
	\label{A:time}
	In this section we give a direct proof  of the formula for $\bT_b(\hf\otimes\hs)$ given in Corollary \ref{cor_simp_obs}.
	Similarly to the direct computation of $\bT_a(\hf\otimes\hs)$, we do so by formulating $bV_f(s)(a,b,c)$ as a linear combination of simple wavelet functions.
	Observe that 
	\[V_{x f}(s)(a,b,c)=\int_{-\infty}^{\infty} s(x) e^{-a/2} c e^{-a}(x-b)\overline{f\big(c e^{-a}(x-b)\big)}dx\]
	\[=c e^{-a}V_f(x s)(a,b,c) - c e^{-a}bV_f(s)(a,b,c).\]
	Thus, 
	\[bV_f(s)(a,b,c)=V_f(x s)(a,b,c)-c e^aV_{x f}(s)(a,b,c).\]
	Next, we show that $c e^aV_{x f}(s)(a,b,c)$ is a simple function.
	\[c e^a V_{x f}(s)(a,b,c) = 
	\int_{-\infty}^{\infty} s(x) e^{-\frac{a}{2}}(x-b)\overline{f\big(c e^{-a}(x-b)\big)} dx.
	\]
	Integrating by parts, setting $v'(x) := s(x)$, 
	$u(x) := e^{-\frac{a}{2}} (x-b)\overline{f\big(c e^{-a}(x-b)\big)}$,
	where the prime sign indicates derivative with respect to $x$, we get
	$v(x):= \int_{-\infty}^{x} s(t)dt$, i.e. $v=\int s$, which is the anti-derivative of $s$, and \[u'(x)=\overline{e^{-\frac{a}{2}} \Big(f\big(c e^{-a}(x-b)\big)+(x-b) c e^{-a} f'\big(c e^{-a}(x-b)\big)\Big)}.\] Thus, 
	\[
	\begin{split}
	& \int_{-\infty}^{\infty} v'(x)u(x)dx \\
	&= vu|_{-\infty}^{\infty} - \int_{-\infty}^{\infty} v(x) u'(x) dx\\
	&= -\int_{-\infty}^{\infty} \Big(\int_{-\infty}^{x} s(t)dt \Big) \overline{e^{-\frac{a}{2}} \Big(f\big(c e^{-a}(x-b)\big)+(x-b) c e^{-a} f'\big(c e^{-a}(x-b)\big)\Big)} dx.   
	\end{split}
	\]
	Note that since $\frac{1}{w}\hs\in \cS$, and since $\hf\in {\rm Dom}_{\cW}(\bT_b)$ implies $x f\in \cS$, the first term vanishes. The second term satisfies  
	\[e^{-\frac{a}{2}}\Big(f\big(ce^{-a}(x-b)\big)+(x-b) ce^{-a} f'\big(c e^{-a}(x-b)\big)\Big) = \pi(a,b,c)(\frac{\partial}{\partial x }x f).\]
	Therefore, the above equals 
	$V_{\frac{\partial}{\partial x}x f}(\int s)$, 
	 and the formula in the frequency domain is
	\[bV_f(s)(a,b,c)=V_{\hf}(i\frac{\partial}{\partial\w}\hs)(a,b,c)-V_{i\w' \frac{\partial}{\partial\w'}\hf}(\frac{1}{\w}\hs)(a,b,c).\]
	Hence, the pull-back to the window-signal space reads
	\[\bT_b(\hf\otimes\hs) = \hf\otimes i\frac{\partial}{\partial\w}\hs - i\w' \frac{\partial}{\partial\w'}\hf\otimes\frac{1}{\w}\hs.\]

\end{appendix}

\subsection*{Acknowledgements}
R.L. acknowledges support by the DFG SPP 1798 “Compressed Sensing
in Information Processing” through Project Massive MIMO-II.

N.S and E.A acknowledge support by the Israeli Ministry of Agriculture’s Kandel Program under grant no. 20-12-0030. Funding was also provided by the framework of ERA-NET Neuron via the Ministry of Health, Israel (\#3-13898).

\end{document}